\documentclass{article}

\PassOptionsToPackage{numbers}{natbib}
    \usepackage[final]{neurips_2020}

\usepackage[T1]{fontenc}    
\usepackage{hyperref}       
\usepackage{url}            
\usepackage{booktabs}       
\usepackage{amsfonts}       
\usepackage{nicefrac}       
\usepackage{microtype}      

\usepackage{graphicx}
\usepackage{amsmath}
\usepackage{amsthm}
\usepackage{amssymb}
\usepackage{booktabs}
\usepackage{multirow}
\usepackage{enumitem}
\usepackage{wrapfig}
\usepackage{algorithm}
\usepackage{algorithmicx}
\usepackage{algpseudocode}
\usepackage{pdflscape}
\usepackage{color}
\newtheorem{thm}{Theorem}
\newtheorem{dfn}{Definition}
\newtheorem{rmk}{Remark}

\usepackage{dcolumn}

\usepackage{array}
\newcommand{\PreserveBackslash}[1]{\let\temp=\\#1\let\\=\temp}
\newcolumntype{C}[1]{>{\PreserveBackslash\centering}p{#1}}
\newcolumntype{R}[1]{>{\PreserveBackslash\raggedleft}p{#1}}
\newcolumntype{L}[1]{>{\PreserveBackslash\raggedright}p{#1}}

\newcommand{\mca}[1]{\multicolumn{1}{c}{\textbf{#1}}}
\newcommand{\mcb}[1]{\multicolumn{2}{c}{\textbf{#1}}}
\newcommand{\mcc}[1]{\multicolumn{3}{c}{\textbf{#1}}}

\newcommand{\vu}{{\vec{u}}}
\newcommand{\vv}{{\vec{v}}}
\newcommand{\vw}{{\vec{w}}}
\newcommand{\vx}{{\vec{x}}}
\newcommand{\vh}{{\vec{h}}}
\newcommand{\vk}{{\vec{k}}}
\newcommand{\vp}{{\vec{p}}}
\newcommand{\vq}{{\vec{q}}}
\newcommand{\bnabla}{{\overline{\nabla}}}
\renewcommand{\d}{{\mathrm{d}}}
\newcommand{\bd}{{\overline{\d}}}
\newcommand{\bG}{{\overline{G}}}
\newcommand{\R}{{\mathbb{R}}}

\newcommand{\up}[1]{$\uparrow$\hspace*{#1}}
\newcommand{\down}[1]{$\downarrow$\hspace*{#1}}
\newcommand{\upl}[1]{\hspace*{#1}$\uparrow$}
\newcommand{\downl}[1]{\hspace*{#1}$\downarrow$}
\newcommand{\score}[3]{$#1${\scriptsize$\pm#2$} ($\times10^{#3}$)}
\newcommand{\pp}{\raisebox{.2\height}{\scalebox{.8}{++}}}

\title{Deep Energy-Based Modeling of\\Discrete-Time Physics}

\author{%
  Takashi Matsubara\\
  Osaka University\\
  Osaka, Japan 560--8531\\
  \texttt{matsubara@sys.es.osaka-u.ac.jp} \\
  \And
  Ai Ishikawa \\
  Kobe University \\
  Kobe, Japan 657--8501\\
  \texttt{a-ishikawa@stu.kobe-u.ac.jp} \\
  \AND
  Takaharu Yaguchi \\
  Kobe University \\
  Kobe, Japan 657--8501\\
  \texttt{yaguchi@pearl.kobe-u.ac.jp} \\
}

\begin{document}

\maketitle

\begin{abstract}
    Physical phenomena in the real world are often described by energy-based modeling theories, such as Hamiltonian mechanics or the Landau theory, which yield various physical laws. Recent developments in neural networks have enabled the mimicking of the energy conservation law by learning the underlying continuous-time differential equations. However, this may not be possible in discrete time, which is often the case in practical learning and computation. Moreover, other physical laws have been overlooked in the previous neural network models. In this study, we propose a deep energy-based physical model that admits a specific differential geometric structure. From this structure, the conservation or dissipation law of energy and the mass conservation law follow naturally. To ensure the energetic behavior in discrete time, we also propose an automatic discrete differentiation algorithm that enables neural networks to employ the discrete gradient method.
\end{abstract}

\section{Introduction}
Deep neural networks have achieved significant results for a variety of real-world tasks such as image processing~\cite{He2015a,Zhu2017}, natural language processing~\cite{Devlin2018}, and game playing~\cite{Silver2017}.
Their successes depend on hard-coded prior knowledge, such as translation invariance in image recognition~\cite{LeCun1998} and the manifold hypothesis in data modeling~\cite{Rifai2011}.
The prior knowledge guarantees a desirable property of the learned function.
The Hamiltonian neural network (HNN)~\cite{Greydanus2019} implements the Hamiltonian structure 
on a neural network and thereby produces the energy conservation law in physics.
After its great success, neural networks specifically designed for physical phenomena have received much attention. They have been intensively extended to various forms, such as the Hamiltonian systems with additional dissipative terms~\cite{Zhong2020a}.

Meanwhile, most previous studies aimed to model continuous-time differential equations and employed numerical integrators (typically, an explicit Runge--Kutta method) to integrate the neural network models for learning and computing the dynamics~\cite{Chen2018e,Chen2020a,Greydanus2019,Zhong2020}.
Surprisingly, our numerical experiments reveal that a higher-order numerical integrator with adaptive time-stepping is quite often inferior in performance as compared to a quantitatively lower order but qualitatively superior numerical integrator.
This is because higher-order integrators aim to reproduce continuous-time dynamics while practical learning and computation are in discrete time.
In this case, the qualitative features that the integrators equipped with could be actually essential.

From this point of view, this study proposes a \textit{deep energy-based discrete-time physical model}, which combines neural networks and discrete-time energy-based modeling.
The key ingredient is the structure-preserving integrators, in particular, the discrete gradient method along with the newly-developed automatic discrete differentiation.
In addition, our framework unifies and also extends the aforementioned previous studies.
The main contributions include:

\paragraph*{Applicable to general energy-based physical models.}
Our framework is applicable to general physical phenomena modeled by the energy-based theory, such as Hamiltonian mechanics, the Landau theory, and the phase field modeling.
Our target class includes a Hamiltonian system composed of position and momentum (a so-called natural system, such as a mass-spring system), a natural system with friction, a physical system derived from free-energy minimization (e.g., phase transitions), and a Hamiltonian partial differential equation (PDE) (e.g., the Korteweg--de Vries (KdV) equation and the Maxwell equation).
All equations can be written as a geometric equation.
Most studies have focused on one of the first two systems~\cite{Greydanus2019,Zhong2020a,Zhong2020} under special conditions~\cite{Chen2020a,Saemundsson2020,Tong2020}, or they are too general to model the conservation and dissipation laws~\cite{Chen2018e,Raissi2018}.
The details of the proposed framework along with the target class of the equations and the geometric aspects are described in Section~\ref{sec:physical_systems}.

\begin{wrapfigure}{r}{0.57\textwidth}
    \vspace*{-7mm}
    \includegraphics[scale=0.4,page=1]{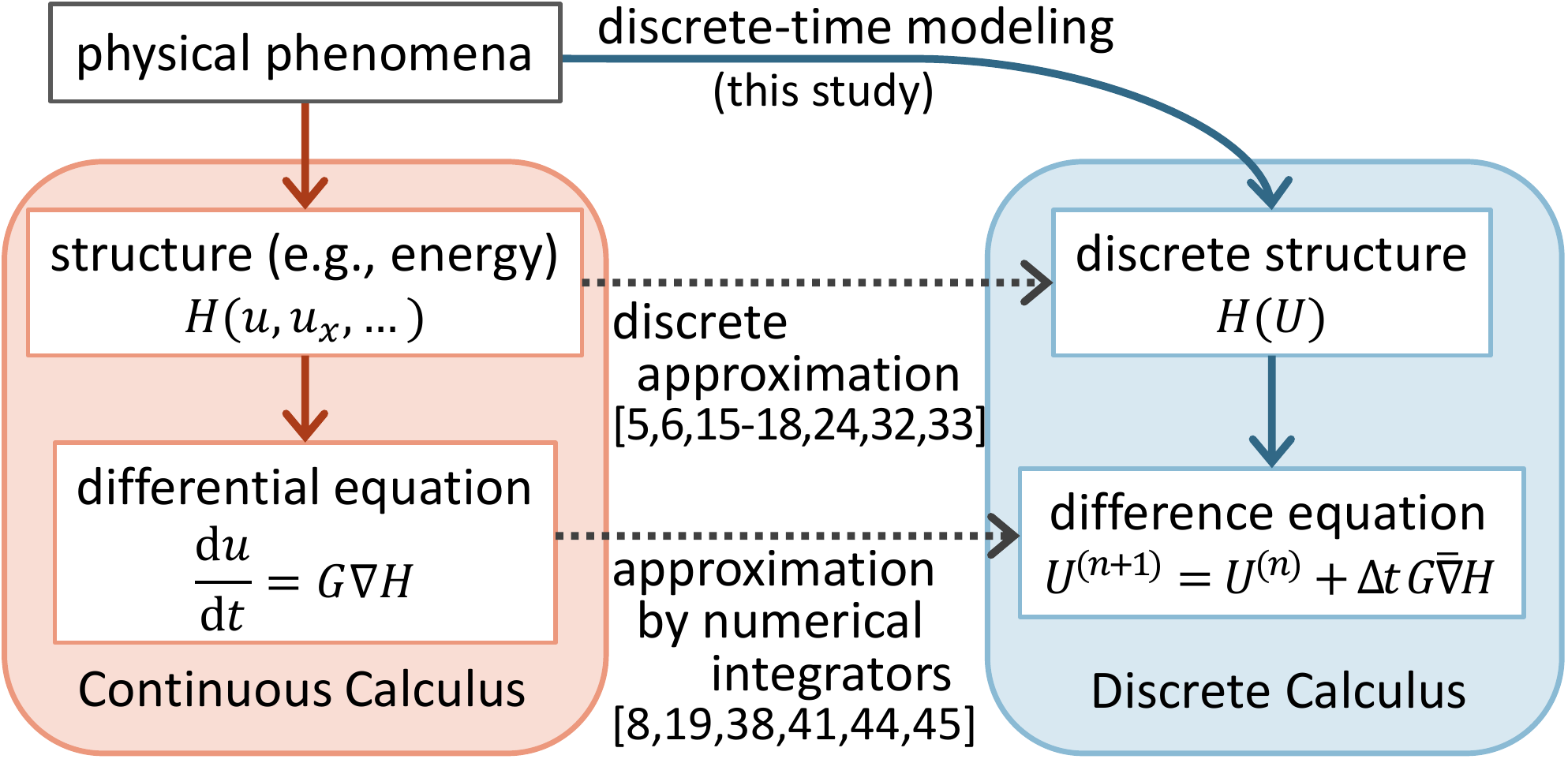}
    \vspace*{-5mm}
    \caption{Modeling based on energy-based theories.}\label{fig:modeling_concept}
    \vspace*{-3mm}
\end{wrapfigure}
\paragraph*{Equipping with the laws of physics in discrete time.}
Previous models interpolate the discrete-time data using numerical integrators for learning and computing~\cite{Chen2020a,Greydanus2019,Saemundsson2020,Tong2020,Zhong2020a,Zhong2020}. The discretization may destroy the geometrical structure from which the laws of physics follow (see the lower part of Fig.~\ref{fig:modeling_concept}).
Conversely, our approach, in principle, learns a discrete-time model from the discrete-time data without the time-consuming interpolation and discretization error (see the upper part).
Using the \emph{discrete gradient}, our approach admits the important laws of physics, particularly the energy conservation or dissipation law and the mass conservation law in discrete time.
We demonstrate this property theoretically in Section~\ref{sec:discrete_gradient} and experimentally in Section~\ref{sec:experiments}.

\paragraph*{Easy-to-use.}
Our approach is based on the discrete gradient method~\cite{Furihata1999,Gonzalez1996,Quispel1996}.
Most discrete gradients require the explicit form of the function (see the middle part of Fig.~\ref{fig:modeling_concept}); hence, they are unavailable for neural networks (see Appendix~\ref{appendix:GI} for reference).
We propose an automatic discrete differentiation algorithm, which automatically obtains the discrete gradient of the neural networks composed of linear and nonlinear operations.
The proposed algorithm can be implemented in a similar way to the current automatic differentiation algorithm~\cite{Griewank2008}; we provide it as a PyTorch library~\cite{Paszke2017}%
\footnote{\url{https://github.com/tksmatsubara/discrete-autograd}}.
We introduce the detailed algorithm in Section~\ref{sec:add}.

\section{Related Work}
\paragraph{Neural Networks for Differential Equations.}
Since the 1990s, many studies have attempted to approximate ordinary differential equations (ODEs) and PDEs by applying neural networks ~\cite{Anastassi2014,Cichock1992,Lagaris1998,Raissi2018,Ramuhalli2005,Rudd2014}.
Recent advances in the automatic differentiation algorithm~\cite{Griewank2008} have enabled us to build more complicated neural network architectures.
Neural ODE (NODE)~\cite{Chen2018e} has re-established neural networks for modeling ODEs.
NODE treats the output of a time-dependent neural network as the time derivative of the input; thereby, defining an ODE in a general way.
Moreover, NODE employs numerical integrators to train and integrate the neural network model.
Several studies attempted to model a PDE system using regularization terms to mimic the conservation laws~\cite{Raissi2018,Wu2020a}.
They were insufficient to ensure the conservation laws in physical systems.

The HNN approximates an energy function $H$ from the data using a neural network, and thereby, builds a Hamiltonian system~\cite{Greydanus2019}.
The time-derivative of the states $(\vq,\vp)$ is given using the gradient $\nabla H$ of the energy $H$, which is called the Hamiltonian, specifically, $\d\vq/\d t=\nabla_\vp H$ and $\d\vp/\d t=-\nabla_\vq H$, where $\vq$ and $\vp$ denote the position and momentum, respectively.
Following the HNN, the symplectic ODE-Net integrates an HNN-like model using a Runge--Kutta method; thus, enabling learning from the discrete-time data~\cite{Zhong2020}.
The dissipative SymODEN generalized it to a model with friction and input~\cite{Zhong2020a}.
We summarized the previous studies in Table~\ref{tab:comparison}.

\begin{table}[t]
    \centering\small
    \caption{Comparison with Other Studies}\label{tab:comparison}
    \begin{tabular}{lcccccccc}
        \toprule
        \textbf{}                                & HNN                  & SymODEN          & Dissipative       & SRNN/VIN                                  & DGNet        \\
        \textbf{}                                & \cite{Greydanus2019} & \cite{Zhong2020} & \cite{Zhong2020a} & \cite{Chen2020a,Saemundsson2020,Tong2020} & (this paper) \\
        \midrule
        Hamiltonian system                       & yes                  & yes              & yes               & yes                                       & yes          \\
        Dissipative ODE                          &                      &                  & yes               &                                           & yes          \\
        \midrule
        Hamiltonian PDE                          &                      &                  &                   &                                           & yes          \\
        Dissipative PDE                          &                      &                  &                   &                                           & yes          \\
        \midrule
        Learning from finite difference          &                      & approx.$^*$      & approx.$^*$       & approx.$^*$                               & yes          \\
        Strict conservation law in discrete-time &                      &                  &                   & approx.$^{**}$                            & yes          \\
        Strict dissipation law in discrete-time  &                      &                  &                   &                                           & yes          \\
        \bottomrule
        \multicolumn{6}{l}{$^*$ Interpolating by numerical integrators. $^{**}$ Conserving only the ``shadow'' Hamiltonian.}                                              \\
    \end{tabular}
    \vspace*{-3mm}
\end{table}

\paragraph{Structure-Preserving Numerical Methods.}
Most differential equations that arise as models for physical phenomena admit some laws of physics, e.g., the energy and other conservation laws of the Hamilton equation and the mass conservation law and energy dissipation properties of the equations for phase-transition phenomena.
Numerical integrators that reproduce those properties are called structure-preserving integrators or geometric integrators~\cite{Hairer2006}.

The aforementioned studies mainly employed classical Runge--Kutta methods for numerical integration, which in general destroy these properties~\cite{Hairer2006}.
Several recent studies have employed symplectic integrators, which conserve an approximated energy called a ``shadow'' Hamiltonian in discrete time~\cite{Chen2020a,Saemundsson2020,Tong2020}.
These studies considered only the systems of which the Hamiltonian $H$ is separable, i.e., expressible as the sum of the potential and kinetic energies.
This is quite restrictive; in fact, most of the important Hamiltonian PDEs (e.g., the shallow water equations and the nonlinear Schr\"odinger equation) are not in this class.
Moreover, structure-preserving integrators for dissipative systems have never been employed.
This is because these integrators are often based on the discrete gradient method; however, no efficient discrete gradient has been available for neural networks.

Several studies have focused on Lagrangian mechanics~\cite{Cranmer2020,Saemundsson2020}.
Lagrangian mechanics can be expressed using the time derivative of the position, while the Hamiltonian mechanics requires conjugate momentum.
The main drawback is that it is not obviously extendable to general dissipative systems.
We consider it out of scope of this study, but the proposed method is extendable to it~\cite{Yaguchi2013}.

\section{Methods}\label{sec:method}
\subsection{General Form of Energy-Based Dynamical Systems}\label{sec:physical_systems}
We focus on the following formulation of the models by the energy-based theories, which expresses a wide variety of physical systems described by ODEs and discretized PDEs~\cite{Furihata1999, Quispel1996a}.
The system has a state $\vu\in\R^N$ and an energy function $H:\R^N\rightarrow \R$.
The time evolution is expressed as 
\begin{equation}
    \textstyle\frac{\d \vu}{\d t}=G(\vu) \nabla H(\vu),\label{eq:gradient_flow}
\end{equation}
where $G\in\R^N\times\R^N$ is a matrix, which can be state-dependent, and $\nabla H(\vu)$ is the gradient of the system energy $H$ with respect to the state $\vu$.
Systems of this form arise as differential geometric equations on Riemannian or symplectic manifolds.
See Appendix \ref{appendix:geometry} for reference.
The system $(H,G,\vu)$ has the following laws of physics.
\begin{thm}\label{thm:continuous_conservation_dissipation}
    The system has the energy dissipation law if $G\leq O$ and the energy conservation law if $G$ is skew-symmetric.
\end{thm}
See Appendix~\ref{appendix:proofs} for the proofs of the theorems for this study.
Note that $G\le O$ denotes that the matrix $G$ is negative semi-definite, with which $\vx^\top G \vx\le 0$ for any vector $\vx$.
A matrix $G$ is skew-symmetric if $G^\top = - G$, and then $\vx^\top G \vx=0$ for any vector $\vx$.

\begin{thm}\label{thm:mass_conservation}
    The system has the mass conservation law in the sense that $\d (\sum_k u_k) /\d t = 0$ if the vector $\vec{1} = (1,1,\ldots,1)$ is in the left kernel of $G$ (i.e., $\vec{1} G = \vec{0}$).
\end{thm}
Thus, we can design the neural network models with the above laws of physics by defining the models $(H,G,\vu)$, where $G$ satisfies the required conditions for the laws of physics shown in the above theorems and $H$ is designed by a neural network.

\begin{rmk}
    The models $(H,G,\vu)$ with $H$ represented by neural networks widely extend the scope of the previous studies.
    In particular, the discretized-in-space PDEs (e.g., the KdV equation~\cite{KdV1895} and {\color{cyan} the} Cahn--Hilliard equation~\cite{Cahn1958}) have not been treated like this before.
    This is a significant contribution in this study.
\end{rmk}

A natural system is a Hamiltonian system associated to a Hamiltonian function $H$ that is the sum of the potential and kinetic energies.
This is expressed as the system $(H,G=S,\vu)$ for the matrix
\begin{equation}
    S =
    {\scriptsize\setlength\arraycolsep{1pt}
    \renewcommand{\arraystretch}{.4}
    \begin{pmatrix}
        O    & I_n \\
        -I_n & O
    \end{pmatrix}
    },\label{eq:hamiltonian_system}
\end{equation}
where $2n=N$ and $I_n$ denotes an $n$-dimensional identity matrix.
The first $n$ elements of the state $\vu$ denote the position $\vq$ and the remaining denotes the momentum $\vp$.
The matrix $S$ is skew-symmetric, and the system $(H,G=S,\vu)$ conserves the system energy $H$.
A pendulum, a mass-spring system, and N-body problems are expressible by this form.
Besides, the system $(H,G=S-R,\vu)$ expresses a natural system with friction when $S$ is the one shown above and
\begin{equation}
    R = \mathrm{diag}(0 \dots 0\ g_{1} \dots g_{n}), \label{eq:dissipative_system}
\end{equation}
where $g_k\ge 0$ is a friction term that dampens the momentum $p_k$; thus, dissipating the system energy $H$ because $(S-R)\le O$.
Most previous studies focused on these two types of systems~\cite{Chen2020a,Greydanus2019,Saemundsson2020,Zhong2020a,Zhong2020}.

From a geometric point of view, the matrix $G$ in the above form means that the systems are defined on cotangent bundles, while the following approach is formulated on general symplectic or Riemannian manifolds, enabling our method to handle the various PDE systems~\cite{Hairer2006}.
In fact, the formulation $(H, G, \vu)$ can express the discretized PDE systems.
For example, PDEs under the periodic boundary condition can be discretized by using the central difference operators, of which the matrix representations are as follows.
\begin{equation}
    D=\frac{1}{2\Delta x}
    {\scriptsize\setlength\arraycolsep{1pt}
        \renewcommand{\arraystretch}{.4}
        \begin{pmatrix}
            0  & 1 &        &    & -1 \\
            -1 & 0 & 1                \\
               &   & \ddots           \\
               &   & -1     & 0  & 1  \\
            1  &   &        & -1 & 0  \\
        \end{pmatrix}
    },\quad
    D_2=\frac{1}{(\Delta x)^2}
    {\scriptsize\setlength\arraycolsep{1pt}
        \renewcommand{\arraystretch}{.4}
        \begin{pmatrix}
            -2 & 1  &        &    & 1  \\
            1  & -2 & 1                \\
               &    & \ddots           \\
               &    & 1      & -2 & 1  \\
            1  &    &        & 1  & -2 \\
        \end{pmatrix}
    },\label{eq:first_second_order_difference}
\end{equation}
where $\Delta x$ is the space mesh size.
The matrices $D$ and $D_2$ 
represent first--order and second--order central differences, respectively.
The $k$-th element $u_k$ of the state $\vu$ corresponds to the mass at the position $x=k \Delta x$, and the systems $(H,G=D,\vu)$ and $(H,G=D_2,\vu)$ admit the mass conservation law.
For suitable discretization of general differential operators, see Appendix \ref{appendix:semi-discretization}.
The system $(H,G=D,\vu)$ is a Hamiltonian PDE, which includes the shallow water equations such as the KdV equation, the advection equation, and the Burgers equation~\cite{Burgers1948}.
The matrix $D$ is skew-symmetric; hence, the system $(H,G=D,\vu)$ conserves the energy $H$.
The system $(H,G=D_2,\vu)$ expresses a physical system derived from the Landau free-energy minimization including the Cahn--Hilliard equation and the phase--field model for the phase transitions and the pattern formulations.
The energy $H$ dissipates because $D_2 \leq O$.
Other target equations include the equations with complex state variables, such as the Schr\"odinger equation and the Ginzburg--Landau equation.
See \cite{Furihata1999} for details.


\subsection{Discrete Gradient for Energetic-Property-Preserving Integration}\label{sec:discrete_gradient}
The discrete gradient is defined as the following vector-valued function~\cite{Furihata1999,Gonzalez1996,Itoh1988,Quispel1996}.
\begin{dfn}\label{def:discrete_gradient}
    For $H:\R^N\rightarrow \R$, $\bnabla H:\R^N \times \R^N \to \R^N$ that satisfies the following conditions is called a discrete gradient of $H$:
    \begin{equation}\label{eq:discrete_gradient}
        \textstyle H(\vu) - H(\vv) = \bnabla H(\vu, \vv) \cdot (\vu - \vv), \quad \bnabla H(\vu, \vu) = \nabla H(\vu),
    \end{equation}
    where $\cdot$ denotes an inner product.
\end{dfn}
The first condition corresponds to the chain-rule $\d H(\Delta\vu;\vu) = \nabla H(\vu)\cdot\Delta\vu$ for the Fr\'echet derivative $\d H(\cdot;\vu)$ of $H$ at $\vu$ and an infinitesimal change $\Delta\vu$ of $\vu$.
The second condition verifies that the discrete gradient $\bnabla H$ is certainly an approximation of the gradient $\nabla H$.
The inner product is typically the standard Hermitian inner product for ODEs and the discrete $L^2$ inner product $\langle \vu, \vv \rangle_{L^2_{\mathrm{d}}} := \sum u_k v_k \Delta x$ for discretized PDEs.

With the discrete gradient $\bnabla H$, a discrete analogue of the system in Eq.~\eqref{eq:gradient_flow} is expressed as follows.
\begin{equation}
    \frac{\vu^{(n+1)} - \vu^{(n)}}{t^{(n+1)}-t^{(n)}}
    = \bG(\vu^{(n+1)}, \vu^{(n)})  \bnabla H(\vu^{(n+1)}, \vu^{(n)}),\label{eq:discrete_system}
\end{equation}
where $\vu^{(n)}$ denotes the state $\vu$ at time $t^{(n)}$.
The matrix $\bG$ is an approximation to $G$ that satisfies the conditions of Theorem \ref{thm:continuous_conservation_dissipation} and/or \ref{thm:mass_conservation} required by the target system.
\begin{thm}\label{thm:discrete_conservation_dissipation}
    The discrete system in Eq.~\eqref{eq:discrete_system} has the discrete energy dissipation law if $\bG \leq O$ and the discrete energy conservation law if $\bG$ is skew-symmetric.
    In particular, if the system is dissipative, the amount of energy dissipation is an approximation of that of the continuous system.
    The system has the discrete mass conservation law if the vector $\vec{1} = (1,1,\ldots,1)$ is in the left kernel of $\bG$.
\end{thm}
A discrete gradient $\bnabla H$ is not uniquely determined; hence, several methods have been proposed so far~\cite{Celledoni2012}.
However, most methods are inapplicable to neural networks because they require a manual deformation of the system equation~\cite{Furihata1999}.
See Appendix \ref{appendix:GI} for details.

A conceptual comparison between discrete gradient methods and symplectic integrators~\cite{Chen2020a,Saemundsson2020,Zhong2020} is summarized in Appendix~\ref{appendix:comparison_with_symplectic}.

\subsection{Automatic Discrete Differentiation Algorithm}\label{sec:add}
To obtain a discrete gradient $\bnabla H$ of the neural networks, we propose the \emph{automatic discrete differentiation} algorithm as an extension of the automatic differentiation algorithm~\cite{Griewank2008}.
Preparatorily, we introduce a discrete differential $\bd H$, which is a discrete counterpart of the Fr\'echet derivative $\d H$~\cite{Celledoni2014};
\begin{dfn}\label{def:discrete_differential}
    A discrete differential $\bd H : \R^N\times\R^N\times\R^N\rightarrow\R^M$ of a function $H:\R^N\rightarrow\R^M$ is a function that satisfies the following conditions;
    \begin{equation}
        \bd H(a\vx;\vv,\vu)=a\bd H(\vx;\vv,\vu),\ H(\vv)-H(\vu)=\bd H(\vv-\vu;\vv,\vu),\ \bd H(\cdot;\vu,\vu)=\d H(\cdot;\vu),
    \end{equation}
    for a scalar value $a$ and the Fr\'echet derivative $\d H(\cdot;\vu)$ of $H$ at $\vu$.
\end{dfn}
For a discrete differential $\bd H$ of a function $H:\R^N\rightarrow\R$, there exists a discrete gradient $\bnabla H$ such that $\bnabla H(\vv,\vu)\cdot \vw=\bd H(\vw;\vv,\vu)$.
This relationship is obvious from Definitions~\ref{def:discrete_gradient} and \ref{def:discrete_differential}, and it is a discrete analogue of the chain-rule $\nabla H(\vu)\cdot \vw=\d H(\vw;\vu)$.

Our proposal is to obtain a discrete differential $\bd H$ of the neural network model $H$ using the automatic discrete differentiation algorithm, and thereby, a discrete gradient $\bnabla H$.
The automatic differentiation algorithm depends on the chain rule, product rule, and linearity.
For the functions $f:\R\rightarrow\R$ and $g:\R\rightarrow\R$, it holds that
\begin{equation}
    \textstyle \frac{\partial}{\partial x}(f\circ g)=\frac{\partial f}{\partial g}\frac{\partial g}{\partial x},\ \
    \frac{\partial}{\partial x}(fg)=g\frac{\partial f}{\partial x}+f\frac{\partial g}{\partial x},\ \
    \frac{\partial}{\partial x}(f+ g)=\frac{\partial f}{\partial x}+\frac{\partial g}{\partial x}.
    \label{eq:rules}
\end{equation}
\begin{thm}\label{thm:discrete_differential_rules}
    For any $x_1,x_2,\Delta x\in\R$ and functions $f:\R\rightarrow\R$ and $g:\R\rightarrow\R$, the chain-rule, product rule, and linearity for the discrete differential are respectively expressed as
    \begin{equation}
        \begin{split}
            \bd (f\circ g)(\Delta x;x_1,x_2)
            &= \bd f(\cdot;g(x_1),g(x_2))\circ \bd g(\Delta x;x_1,x_2),\\
            \bd (fg)(\Delta x;x_1,x_2) &\textstyle=\frac{g(x_1)+g(x_2)}{2}\bd f(\Delta x;x_1,x_2)+\frac{f(x_1)+f(x_2)}{2}\bd g(\Delta x;x_1,x_2),\\
            \bd (f+g)(\Delta x;x_1,x_2)&=\bd f(\Delta x;x_1,x_2)+\bd g(\Delta x;x_1,x_2).
        \end{split}
    \end{equation}
\end{thm}

For any linear operations such as the fully-connected and convolution layers, a discrete differential is equal to the Fr\'echet derivative because of the linearity.
For an element-wise nonlinear activation function $f:\R\rightarrow\R$, we employed the following discrete differential~\cite{Gonzalez1996}.
\begin{equation}
    \bd f(\Delta x;x_1,x_2)=\begin{cases}
        \frac{f(x_1)-f(x_2)}{x_1-x_2}\Delta x & \mbox{if } x_1\neq x_2 \\
        \d f(\Delta x;\frac{x_1+x_2}{2})      & \mbox{otherwise}.
    \end{cases}\label{eq:discrete_differential_of_activation}
\end{equation}
The product rule is applicable to bilinear operations such as attention, graph convolution, transformer, and metric function~\cite{Devlin2018,Faghri2018}.

Given the above, we propose the automatic discrete differentiation algorithm.
With the algorithm, one can automatically obtain a discrete differential $\bd H$ of a neural network $H$ given two arguments, which is then converted to a discrete gradient $\bnabla H$.
The computational cost is no more than twice of the ordinary automatic differentiation.
The algorithm is applicable to any computational graph such as convolutional neural network~\cite{Greydanus2019} and graph neural network~\cite{Desai2020}, and thereby one can handle extended tasks or further improve the modeling accuracy.
For reference, we introduce the case with a neural network that is composed of a chain of functions in Algorithm~\ref{alg:add} in Appendix~\ref{app:add}.
We call a neural network obtaining a discrete gradient $\bnabla H$ by using the automatic discrete differentiation algorithm \emph{DGNet}, hereafter.

\subsection{Learning and Computation by the Discrete-Time Model}
Using DGNet, we propose a deep energy-based discrete-time physical model
that can learn from the discrete-time data directly as follows.
Given a time series, DGNet accepts two state vectors $u^{(n)}$ and $u^{(n+1)}$ at time steps $n$ and $n+1$, and then it outputs two scalar system energies $H(\vu^{(n)})$ and $H(\vu^{(n+1)})$.
The discrete gradient $\bnabla H(\vu^{(n+1)}, \vu^{(n)})$ is obtained by the automatic discrete differentiation algorithm. 
The model is trained to minimize the squared error between the left- and right-hand sides of Eq.~\eqref{eq:discrete_system};
\begin{equation}
    \textstyle\mbox{minimize\ }
    \sum_n
    \|
    \frac{\vu^{(n+1)} - \vu^{(n)}}{t^{(n+1)}-t^{(n)}}
    - \bG(\vu^{(n+1)}, \vu^{(n)}) \bnabla H(\vu^{(n+1)}, \vu^{(n)})\|_2^2\label{eq:discrete_differential_objective}
\end{equation}
Then, the error is back-propagated through the computational graphs including the neural network model, the discrete gradient, and the matrix $\bG$ by the ordinary automatic differentiation algorithm.
For training, the computational cost of the proposed scheme in Eq.~(6) is no more than twice of the HNN with the Euler method and typically tens times smaller than that with the adaptive Dormand--Prince method.
Through this learning process, DGNet potentially restores the true gradient $\nabla H$ from the sampled data because the discrete gradient $\bnabla H$ is equal to the true gradient $\nabla H$ when two arguments are equal by Definition~\ref{def:discrete_gradient}.

For a time-series prediction, DGNet predicts the next state implicitly by solving the implicit scheme in Eq.~\eqref{eq:discrete_system} and conserves the energy strictly.
The proposed discrete gradient $\bnabla H$ is time-symmetric, which implies that the proposed method is at least a second--order method~\cite{Quispel1996}.
Higher-order methods can be designed using the composition method (using multiple sub-steps) and the higher-order temporal difference (using multiple steps) as introduced in \cite{Furihata2010}.

Indeed, the training and prediction can be performed in a different manner.
After learning from the finite differences, DGNet provides the gradient $\nabla H$ so it is available for an explicit numerical method, which can be more computationally efficient and be preferable when the learned models are used in existing physics simulators (e.g., Matlab).
When the true time-derivative is known, DGNet can learn it as the previous models did.
Then, it can predict the next step using the discrete gradient implicitly while conserving energy.

\section{Learning of Partial and Ordinary Differential Equations}\label{sec:experiments}
\paragraph{Comparative Models.}
We examined the proposed DGNet and comparative methods.
NODE is a neural network that outputs the time-derivative of the states in a general way~\cite{Chen2018e}.
The HNN is a neural network where the output represents the system energy $H$, and its gradient with respect to the input state $\vu$ is used for the time-derivative~\cite{Greydanus2019}.
In our experiments, they were trained from a finite difference between two successive time steps using a numerical integrator, which is similar to some previous studies~\cite{Chen2020a,Saemundsson2020,Tong2020, Zhong2020a, Zhong2020}.
For numerical integrators, we employed the explicit midpoint method (RK2) and the Dormand--Prince method with adaptive time-stepping (ada.~DP); they are second-- and fourth--order explicit Runge--Kutta methods.
Then, the output error was back-propagated through all stages~\cite{Chen2018e}.
In terms of applying the HNN to the discretized PDEs, we generalized it by using the formulation in Section~\ref{sec:physical_systems} and denoted it as the HNN\pp.
DGNet was trained to minimize the objective in Eq.~\eqref{eq:discrete_differential_objective}; for simplicity, the matrix $G$ of the system was assumed to be known, and we used $\bG=G$. 
We also employed explicit numerical integrators for DGNet's prediction to reduce the computational cost from the implicit scheme in Eq.~\eqref{eq:discrete_system}.

\paragraph{Hamiltonian PDE.}
We evaluated the models on a Hamiltonian PDE, namely the KdV equation, which is a famous model that has soliton solutions~\cite{Furihata1999,Furihata2001}.
Of the discretized 1-dimensional KdV equation, the system energy $H$ and time evolution are expressed as follows.
\begin{equation}
    \textstyle H(\vu)=\Delta x\textstyle \sum_k (-\frac{1}{6}\alpha u_k^3-\frac{1}{2}\beta(D \vu)_k^2),\
    \textstyle \frac{\partial \vu}{\partial t}=\textstyle D\nabla H(\vu) = D(-\frac{1}{2}\alpha(\vu\odot\vu)+\beta(D_2 \vu)),\label{eq:kdv}
\end{equation}
where the subscript $_k$ denotes the $k$-th element, $D$ and $D_2$ denote the first-- and second--order central differences in Eq.~\eqref{eq:first_second_order_difference}, and $\odot$ denotes the element-wise multiplication.
The coefficients $\alpha$ and $\beta$ determine the spatio-temporal scales.
We set $\alpha=-6$, $\beta=1$, the spatial size to 10 space units, and the space mesh size $\Delta x$ to 0.2 .
At $t=0$, we set two solitons, each of which were expressed as $-\frac{12}{\alpha} \kappa^2 \mathrm{sech}^2(\kappa (x - d))$.
$\kappa$ denotes the size randomly drawn from $\mathcal U(0.5,2)$, and $d$ denotes its initial location randomly, which is determined to stay 2.0 space units away from each other.
We employed the discrete gradient method in \cite{Furihata2001} to ensure the energy conservation law.
We simulated the equation with a time step size of $\Delta t = 0.001$ for 500 steps and obtained 100 time series (90 for training and 10 for the test).
Every experiment in this section was done with double precision.

We employed a neural network composed of a 1-dimensional convolution layer followed by two fully-connected layers.
A convolution layer with a kernel size of 3 is enough to learn the central difference.
The matrix $G=D$ was implemented as a 1-dimensional convolution layer with the kernel of $(-1/2\Delta x,0,1/2\Delta x)$ and periodic padding.
Following the study on the HNN~\cite{Greydanus2019}, the activation function was the hyperbolic tangent, the number of hidden channels was 200, and each weight matrix was initialized as a random orthogonal matrix.
Each network was trained using the Adam optimizer~\cite{Kingma2014b} with a batch size of 200 and a learning rate of 0.001 for 10,000 iterations.

\begin{table}[t]
    \centering\small
    \caption{Results on the PDE datasets.}\label{tab:pde_score}
    \begin{tabular}{lllrrrrrr}
        \toprule
                                               & \mcb{Integrator}               & \mcc{KdV equation}             & \mcc{Cahn--Hilliard equation}                                                                                     \\
        \cmidrule(lr){2-3}\cmidrule(lr){4-6}\cmidrule(lr){7-9}
        \textbf{Model}                         & \mca{Training}                 & \mca{Prediction}               & \mca{Deriv.}                  & \mca{Energy}  & \mca{Mass}    & \mca{Deriv.}      & \mca{Energy}  & \mca{Mass}    \\
        \midrule
        \multirow{2}{*}{NODE~\cite{Chen2018e}} & RK2                            & RK2                            & >10000                        & >10000        & 2857.81       & 791.25            & >10000        & 914.72        \\
                                               & ada.~\!DP                      & ada.~\!DP                      & >10000                        & >10000        & 2836.45       & 790.48            & >10000        & 913.96        \\
        \midrule
        \multirow{2}{*}{HNN\pp}                & RK2                            & RK2                            & 36.32                         & 6.32          & 0.70          & 344.23            & >10000        & 87.55         \\
                                               & ada.~\!DP                      & ada.~\!DP                      & \underline{23.27}             & 3.01          & 0.34          & \underline{33.03} & 4.89          & 0.80          \\
        \midrule
                                               & \upl{4mm}                      & RK2                            & \up{3mm}                      & 1.84          & 0.28          & \up{2mm}          & >10000        & 821.58        \\
        DGNet                                  & Eq.~\eqref{eq:discrete_system} & ada.~\!DP                      & \textbf{17.48}                & \textbf{1.60} & \textbf{0.25} & \textbf{7.14}     & \textbf{0.34} & \textbf{0.07} \\
                                               & \downl{4mm}                    & Eq.~\eqref{eq:discrete_system} & \down{3mm}                    & \textbf{1.60} & \textbf{0.25} & \down{2mm}        & \textbf{0.34} & \textbf{0.07} \\
        \bottomrule
        \multicolumn{9}{L{13cm}}{\footnotesize The best and second best results are emphasized by bold and underlined fonts, respectively. Multiplied by $10^0$ for Deriv.~and by $10^{\!-\!6}$ for Energy of the Cahn--Hilliard equation, and by $10^{\!-\!3}$ for the others.}
    \end{tabular}
    \vspace*{-3mm}
\end{table}

\begin{wrapfigure}{r}{2.4in}
    \vspace*{-3mm}
    \includegraphics[scale=1.0]{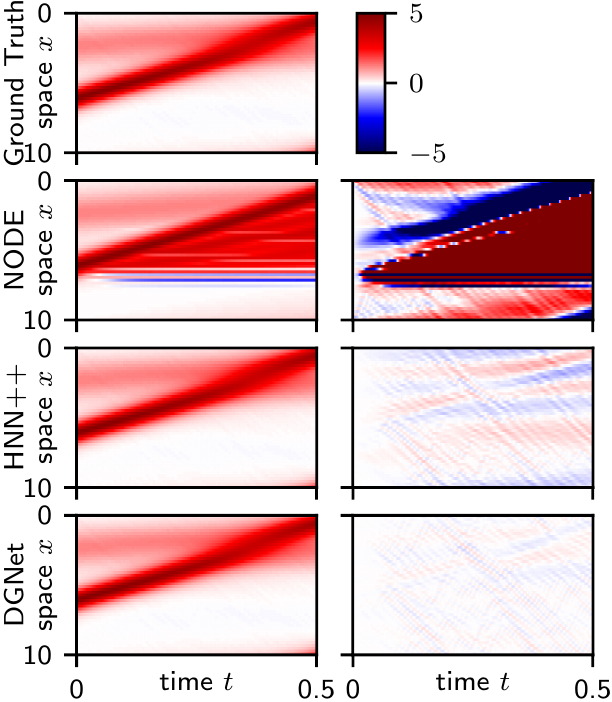}
    \vspace*{-7mm}
    \caption{KdV equation. (left) Predicted state $u$. (right) Error ($\times20$).}\label{fig:kdv}
    \vspace*{-3mm}
\end{wrapfigure}
After training, we examined the mean squared error (MSE) of the time-derivative; we provided an average over 15 trials on Table~\ref{tab:pde_score} (see the column ``Deriv.'').
We omitted the outliers and standard deviations for readability (see Appendix~\ref{appendix:datasets_results} for the full results).
DGNet restored the true time-derivative well.
The HNN\pp\ employed the adaptive Dormand--Prince method, but it suffered from the gap between the time-derivative and the finite difference.
Nonetheless, the application of the HNN to a PDE system is one of the contributions of this study.
NODE failed to model the equation.
For evaluating the long-term consistency, we predicted the test time series from the initial state $\vu^{(0)}$ and obtained the MSE of the total energy and local mass (see the columns ``Energy'' and ``Mass'').
We also visualized the prediction result for each model with the best integrator, which is depicted in Fig.~\ref{fig:kdv}.
DGNet also conserved energy the best with all integrators.
Even though the implicit scheme in Eq.~\eqref{eq:discrete_system} is computationally expensive, DGNet provided the time-derivative for explicit numerical integrators, and it was enough for conserving energy in the present experiment scale (for a longer case, see Appendix~\ref{appendix:datasets_results}).
This result implies that the discrete gradient method provides a good framework for learning from the finite difference; to the best of our knowledge, this is the first time to confirm such contribution of the discrete gradient.
In addition, one might say that the implicit scheme in Eq.~\eqref{eq:discrete_system} is as powerful as the fourth--order integrator with adaptive time stepping even though it is a second--order method.

\paragraph{Dissipative PDE.}
We evaluated the models on a dissipative PDE, namely the Cahn--Hilliard equation.
This equation is derived from free-energy minimization and it describes, for example, the phase separation of copolymer melts~\cite{Furihata1999,Furihata2001}.
The system energy $H$ and time evolution of the discretized 1-dimensional Cahn--Hilliard equation are expressed as follows.
\begin{equation}
    \textstyle H(\vu)=\Delta x\sum_k (\frac{1}{4}(u_k^2-1)^2+\gamma\frac{1}{2}(D \vu)_k^2), \
    \textstyle\frac{\partial \vu}{\partial t}=D_2\nabla H(\vu)=D_2((\vu\!\odot\!\vu\!-\!\vec{1})\!\odot\!\vu-\gamma D_2 \vu),\label{eq:ch}
\end{equation}
where the coefficient $\gamma>0$ denotes the mobility of the monomers.
The mass $u_k$ has an unstable equilibrium at $u_k=0$ (totally melted) and stable equilibria at $u_k=-1$ and $u_k=1$ (totally separated).
We set $\gamma$ to 0.0005, the spatial size to 1, the space mesh size $\Delta x$ to 0.02, the time step size $\Delta t$ to 0.0001, and the initial state $u_k$ to a random sample from $\mathcal U(-0.05,0.05)$.
The other conditions are the same as the case with the KdV equation.

\begin{wrapfigure}{r}{2.4in}
    \vspace*{-7mm}
    \includegraphics[scale=1.0]{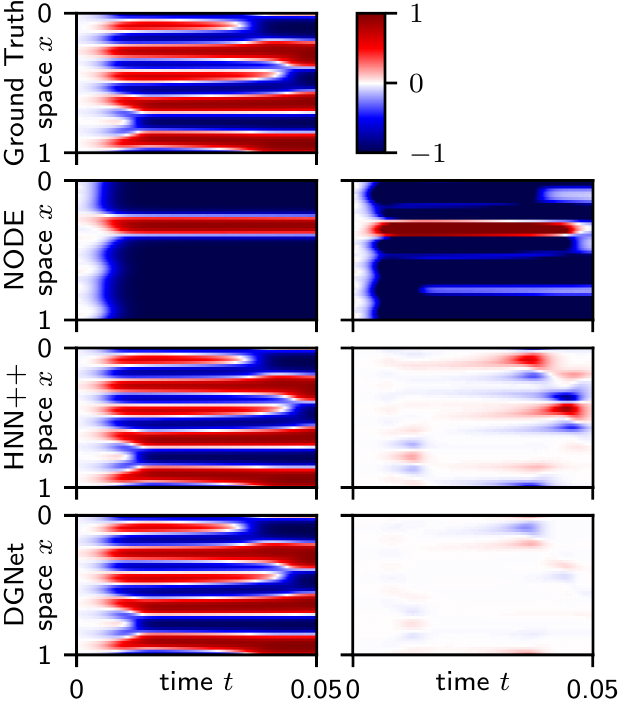}
    \vspace*{-7mm}
    \caption{Cahn--Hilliard equation. (left) Predicted state $u$. (right) Error ($\times3$).}\label{fig:ch}
    \vspace*{-5mm}
\end{wrapfigure}
We summarized the results in Table~\ref{tab:pde_score} and visualized the prediction result for each model with the best integrator in Fig.~\ref{fig:ch}.
DGNet outperformed the HNN\pp\ by a large margin.
The Cahn--Hilliard equation is ``stiff''; this implies that the state can change drastically and an explicit integrator requires a much smaller time step size.
The adaptive Dormand--Prince method evaluated the HNN\pp~50--100 times per time step in the training phase and consumed the proportional computational cost.
However, it did not learn the discrete-time dynamics well; the HNN\pp\ underestimated the diffusion as shown in Fig.~\ref{fig:ch}.
Conversely, DGNet can estimate the dissipative term well, as expected in Theorem~\ref{thm:discrete_conservation_dissipation}.

\begin{table}[t]
    \centering\small
    \caption{Results for the ODE datasets}\label{tab:ode_score}
    \tabcolsep=1.2mm
    \begin{tabular}{lllrrrrrrrr}
        \toprule
                                                  & \mcb{Integrator}               & \mcb{Mass-Spring}              & \mcb{Pendulum}    & \mcb{2-Body}  & \mcb{Real Pendulum}                                                                                                  \\
        \cmidrule(lr){2-3}\cmidrule(lr){4-5}\cmidrule(lr){6-7}\cmidrule(lr){8-9}\cmidrule(lr){10-11}
        \textbf{Model}                            & \textbf{Training}              & \textbf{Prediction}            & \mca{Deriv.}      & \mca{Energy}  & \mca{Deriv.}        & \mca{Energy}      & \mca{Deriv.}     & \mca{Energy}      & \mca{Diff.}      & \mca{Energy}     \\
        \midrule
        \multirow{2}{*}{NODE}                     & RK2                            & RK2                            & 52.68             & 570.32        & 56.67               & 4602.57           & 20.81            & >10000            & \underline{1.38} & 0.62             \\
                                                  & ada.~\!DP                      & ada.~\!DP                      & 55.74             & 574.06        & 55.40               & 4624.66           & 20.71            & >10000            & \textbf{1.37}    & 0.59             \\
        \midrule
        \multirow{2}{*}{HNN~\cite{Greydanus2019}} & RK2                            & RK2                            & \textbf{38.22}    & 61.25         & 42.49               & 404.24            & \underline{5.39} & 93.88             & 1.42             & 2.86             \\
                                                  & ada.~\!DP                      & ada.~\!DP                      & 39.92             & 1.74          & 40.88               & 16.55             & 6.21             & 81.84             & 1.41             & 3.44             \\
        \midrule
        SRNN~\cite{Chen2020a}                     & leapfrog                       & leapfrog                       & 39.47             & 0.69          & \textbf{39.24}      & \underline{11.24} & \textbf{4.36}    & \textbf{40.37}    & (1.38)           & (9.63)           \\
        \midrule
                                                  & \upl{4mm}                      & RK2                            & \up{3mm}          & 61.26         & \up{3mm}            & 743.42            & \up{2mm}         & 81.07             & \up{2mm}         & 0.86             \\
        DGNet                                     & Eq.~\eqref{eq:discrete_system} & ada.~\!DP                      & \underline{38.50} & \textbf{0.62} & \underline{39.30}   & 16.06             & 7.80             & 81.04             & \underline{1.38} & \textbf{0.49}    \\
                                                  & \downl{4mm}                    & Eq.~\eqref{eq:discrete_system} & \down{3mm}        & \textbf{0.62} & \down{3mm}          & \textbf{10.79}    & \down{2mm}       & \underline{81.03} & \down{2mm}       & \underline{0.50} \\
        \bottomrule
        \multicolumn{11}{L{13cm}}{The best and second best results are emphasized by the bold and underlined fonts, respectively. Multiplied by $10^{\!-\!6}$ for the 2-body dataset and by $10^{\!-\!3}$ for the others.}
    \end{tabular}
    \vspace*{-5mm}
\end{table}

\paragraph{Hamiltonian Systems.}
We employed Hamiltonian systems that were examined in the original study of the HNN~\cite{Greydanus2019}, namely a mass-spring system, a pendulum system, and a 2-body system.
Because they are natural systems, we used the matrix $G=S=({\tiny\begin{smallmatrix} 0 & I_n \\ -I_n & 0 \end{smallmatrix}})$.
Instead of the time-derivative, we used the finite difference for training like the cases above.
Moreover, we unified the time step size for training and test (see Appendix~\ref{appendix:datasets_results} for details).
The other experimental settings were the same as the original experiments~\cite{Greydanus2019} and the cases above.
Every experiment of ODEs was done with single precision.
Following the symplectic recurrent neural network (SRNN)~\cite{Chen2020a}, we employed the leapfrog integrator and a pair of networks of the same size to represent the potential energy $V(\vq)$ and kinetic energy $T(\vp)$.
The leapfrog integrator is typically applicable to this class.

We summarized the results in Table~\ref{tab:ode_score}.
DGNet sometimes obtained a worse time-derivative error but it always achieved better prediction errors than the HNN; DGNet learned the contour lines of the Hamiltonian $H$ rather than the time-derivative.
DGNet achieved the best results on the long-term predictions in the mass-spring and pendulum datasets and the second-best result in the 2-body dataset.
The SRNN achieved a remarkable result in the 2-body dataset because its network and integrator are specially designed for the separable Hamiltonian, which is a powerful assumption in general. 
DGNet for the separable Hamiltonian is a possible future study.

\paragraph{Physical System with Friction.}
We evaluated the models on the real pendulum dataset that were obtained by \citet{Schmidt2009} following the study on the HNN ~\cite{Greydanus2019}.
This dataset contains the angle and angular momentum readings of a pendulum bob.
Since the real pendulum has friction, we used the matrix $G=S-R=(\begin{smallmatrix} 0 & 1 \\ -1 & 0 \end{smallmatrix})-(\begin{smallmatrix} 0 & 0 \\ 0 & g \end{smallmatrix})$, where $g$ is an additional parameter that represents the friction and it was initialized to zero.
Solved by a Runge--Kutta method, this model can be regarded as the dissipative symODEN without the control input~\cite{Zhong2020a}.

\begin{wrapfigure}{r}{3.9cm}%
    \centering%
    \vspace*{-3mm}
    \hspace*{-3mm}
    \includegraphics[scale=1.0]{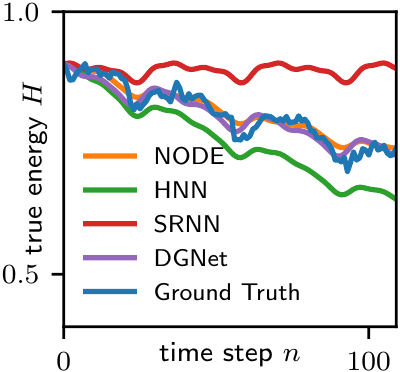}
    \vspace*{-6mm}
    \caption{Results for the real pendulum dataset.}\label{fig:real}
    \vspace*{-4mm}
\end{wrapfigure}
We evaluated the MSE of the finite difference (i.e., 1-step prediction, see the column ``Diff.'') and the MSE of the energies in long-term predictions.
The results are summarized in Table~\ref{tab:ode_score} and Fig.~\ref{fig:real}.
While all methods achieved similar errors in the 1-step prediction, the HNN achieved the worst error in the long-term prediction; the HNN overestimated the friction while DGNet estimated the friction term well, as expected in Theorem~\ref{thm:discrete_conservation_dissipation}.
The energy $H$ derived only from the angle and momentum of the pendulum bob does not monotonically decrease because the other components (e.g., the pendulum rod) are ignored.
DGNet estimated the alternative energy dissipating monotonically, and it predicted the states well.
NODE, which can approximate a general ODE, also worked better than the HNN.
For reference, we confirmed that the SRNN failed in modeling the real pendulum dataset because of the lack of a friction term.

\section{Conclusion}
We proposed a discrete-time energy-based physical model.
The approach unified and widely extended the scope of neural networks for physics simulations.
Using the automatic discrete differentiation algorithm, a neural network is able to use the discrete gradient method, learn from the discrete-time data accurately, and admit the important laws of physics in discrete time.

\section*{Broader Impact}
\paragraph{Novel paradigm of mathematical modeling.}
For computing the physical phenomena, one has to build a difference equation in discrete time.
Mathematical models for physics are typically given as differential equations, and they are discretized using numerical integrators (see the lower part of Fig.~\ref{fig:modeling_concept}).
This discretization may destroy the geometrical structure from which the laws of physics follow.
Most previous studies on neural networks for physical phenomena employ this approach~\cite{Chen2020a,Greydanus2019,Saemundsson2020,Tong2020,Zhong2020a,Zhong2020}.

The discrete gradient method is a discrete-time approximation of a continuous-time structure (see the middle part of Fig.~\ref{fig:modeling_concept})~\cite{Celledoni2012,Celledoni2014,Furihata1999,Furihata2001,Furihata2010,Gonzalez1996,Itoh1988,Quispel1996a,Quispel1996}.
It admits the laws of physics in discrete time, but it suffers from the discretization error, too.
This method has been inapplicable in neural networks until this study. We addressed this issue by introducing the automatic discrete differentiation algorithm.

Our approach is defined in discrete time and it learns discrete-time dynamics directly from discrete-time data (see the upper part of Fig.~\ref{fig:modeling_concept}).
As a result, it never suffers from the discretization error even though the modeling error matters.
In this sense, this study provides a novel paradigm for mathematical modeling.

\paragraph{Novel framework of scientific machine learning.}
The proposed approach combines neural networks and geometric integration, in particular, the discrete gradient method that is derived by the automatic discrete differentiation algorithm.
As far as we know, the proposed framework is the first approach that unifies mathematical modeling from the first principles, data-driven modeling, and energetic-property-preserving numerical computations.
From the viewpoint of scientific computing, the latter two may significantly accelerate scientific simulations.
In practical simulations, modeling and numerical computations have been performed separately, while these must be unified because the results of the simulations often require modification of the mathematical models, and vice versa.

In addition, as implemented by PyTorch, our programming codes for the proposed framework are naturally parallelized.
This implementation is the first numerical library that provides parallelized numerical simulations while using the discrete gradient method, which widely accelerates the computation in scientific simulations.


\begin{ack}
    Funding in direct support of this work: JST CREST Grant Number JPMJCR1914, JST PRESTO Grant Number JPMJPR16EC, JSPS KAKENHI Grant Number 19K20344, 20K11693 and scholoarship by Yoshida Scholarship Foundation.
\end{ack}

{\small

}

\clearpage
\newpage
\appendix
\renewcommand\thetable{A\arabic{table}}
\setcounter{table}{0}
\renewcommand\thefigure{A\arabic{figure}}
\setcounter{figure}{0}
{\huge Supplementary Material: Appendices}
\section{Geometric Numerical Integration}\label{appendix:GI}
Geometric numerical integration is a study on the numerical integrators of ODEs that preserve the geometric property of the target class of equations. Contrary to the classical integrators, such as the Runge--Kutta family of numerical integrators, these integrators are associated with a certain restricted target class of equations; thereby, they are designed so that important geometric properties admitted by the equations in the target class are preserved. As a consequence, the numerical solutions computed by these integrators are not only quantitatively accurate but they are also qualitatively superior to those by the classical integrators.
A typical target class of the equations is the Hamilton equation, which is, in terms of geometry, defined as a symplectic-gradient flow on symplectic manifolds. The Hamilton equation is characterized by the conservation law of a symplectic form, a non-degenerate and closed 2-form (covariant 2-tensor) on the underlying symplectic manifold.
Symplectic integrators are the numerical integrators that preserve this conservation law; hence, they can be in a sense considered as a discrete Hamiltonian system that is an approximation to the target Hamiltonian system.
Due to this property, there should exist a corresponding Hamiltonian function, i.e., energy function, which is called a shadow Hamiltonian and it must be an approximation to the energy function of the target system. Because the shadow Hamiltonian is exactly conserved, the original energy function is conserved not exactly but highly accurately.
This property of symplectic integrators is theoretically guaranteed by the existence theorem of the shadow Hamiltonian.

Because the symplectic integrators conserve only the approximation of the energy function, one may want to design energy-preserving integrators, that is, numerical integrators that preserves the energy function exactly. As such, the discrete gradient method has achieved great success. The Ge--Marsden theorem implies that there does not exist an integrator which is symplectic and energy-preserving~\cite{Marsden1988}.
More precisely, this theorem states that if an integrator is symplectic and energy-preserving for all Hamiltonian systems, then the integrator is completely error-free; the orbits of the numerical results must be exactly on the true orbit.

As shown above, a discrete gradient is defined in Definition \ref{def:discrete_gradient}. Because a discrete gradient is not uniquely determined from this definition, several derivation methods have been proposed. A preferred method is the average vector field method, which is second--order accurate and also conjugate-symplectic (roughly speaking, the method also approximates a symplectic method with higher order accuracy).
However, most of the existing discrete gradients require explicit representation of the Hamiltonian; hence, they are not available for neural networks.
An exception is the Ito--Abe method~\cite{Itoh1988}
\begin{align}
    \bnabla_{\mathrm{IA}} H(\vu, \vv) =
    \begin{pmatrix}
        \frac{
            H((u_1, u_2, \ldots, u_N)^\top)
            -
            H((v_1, u_2, \ldots, u_N)^\top)
        }{
            u_1 - v_1
        } \\
        \frac{
            H((v_1, u_2, \ldots, u_N)^\top)
            -
            H((v_1, v_2, \ldots, u_N)^\top)
        }{
            u_2 - v_2
        }
        \\
        \vdots
        \\
        \frac{
            H((v_1, v_2, \ldots, u_N)^\top)
            -
            H((v_1, v_2, \ldots, v_N)^\top)
        }{
            u_N - v_N
        }
    \end{pmatrix}
\end{align}
which requires a number of evaluations of the energy function; hence, it is computationally intractable for large systems.
In contrast to the Ito--Abe discrete gradient, the discrete gradient obtained by our method is available for neural networks and computationally efficient. It requires no more than twice the computational cost of the ordinary automatic differentiation. Hence, the proposed automatic discrete differentiation algorithm is indispensable for practical application of the discrete gradient method for neural networks.
See also ~\cite{Furihata2010, Hairer2006}.

\section{The Target Equations: Geometric Ordinary and Partial Differential Equations}\label{appendix:geometry}
The target equations for this study are the differential equations with a certain geometric structure. Although the differential equations include PDEs, for simplicity, we only consider ODEs; we assume that the target PDE is semi-discretized in space so that the equation can be transformed into a system of ODEs (see Appendix \ref{appendix:semi-discretization}). Let $\mathcal{M}$ be a finite dimensional manifold and $T \mathcal{M}$ be the tangent bundle of $\mathcal{M}$.
A covariant 2-tensor $\omega_{\vec{u}}$ at $\vec{u} \in \mathcal{M}$ is a bi-linear continuous map $T_{\vec{u}} \mathcal{M} \times T_{\vec{u}} \mathcal{M} \to \mathbb{R}$, that is, a bi-linear continuous map that maps two tangent vectors at $\vec{u}$ into a real number.
We assume that $\omega_\vu$ is non-degenerate in the sense that for any bounded linear map $f \in T_\vu^* \mathcal{M}$, $\omega_{\vu}(\vv, \cdot) = f(\cdot)$ defines a unique vector $\vv \in T_\vu \mathcal{M}$. The target class of the differential equations for this study is the equations of the following form.
\begin{equation}\label{eq:target_geom}
    \frac{\d \vu}{\d t}=X, \quad \omega_{\vu}(\vec{X}, \vv) = \d H(\vv) \quad \mbox{for all}\ \vv \in T_\vu \mathcal{M},
\end{equation}
where $H: \mathcal{M} \to \mathbb{R}$ is an energy function and $\d H$ is the Fr\'echet derivative of $H$.
In fact, because $\omega$ is non-degenerate, the vector $\vec{X}$ in the above equation is uniquely determined.
The covariant 2-tensor $\omega_{\vu}$ can be written as follows.
\begin{equation}
    \omega_{\vu}(\vv, \vw) = \vw^\top A(\vu) \vv
\end{equation}
with a matrix $A(\vu)$; hence, Eq.~\eqref{eq:target_geom} is shown to be equivalent to
\begin{equation}
    \vw^\top A(\vu) \frac{\d \vu}{\d t}= \d H(\vw).
\end{equation}
By using the standard inner product $\left \langle \cdot, \cdot \right \rangle$, this can be expressed as follows.
\begin{equation}
    \left \langle A(\vu) \frac{\d \vu}{\d t}, \vw \right \rangle = \left \langle \nabla H, \vw \right \rangle,
\end{equation}
from which it follows
\begin{equation}
    \frac{\d \vu}{\d t} = G(\vu) \nabla H, \quad G(\vu) = A(\vu)^{-1},
\end{equation}
where $A(\vu)^{-1}$ exists because $\omega_{\vu}$ is non-degenerate. This is our target equation in Eq.~\eqref{eq:gradient_flow}.

The typical examples of the manifolds with such a 2-tensor are the Riemannian manifold \cite{Celledoni2018} and the symplectic manifold \cite{Marsden1999}. In the former case, $\omega_{\vu}$ is the inverse of the matrix that represents the metric tensor, which corresponds to the negative definite matrix $G(\vu)$. In the latter case, $\omega_{\vu}$ is the symplectic form, for which the matrix $G(\vu)$ is skew-symmetric.
They correspond to these manifolds in the target equation in Eq.~\eqref{eq:target_geom} and they are known as the gradient flow and the symplectic flow, respectively.

\section{Proofs}\label{appendix:proofs}
This section provides the proofs of the Theorems in the main text.
\begin{proof}[Proof of Theorem \ref{thm:continuous_conservation_dissipation}]
    From the chain-rule, it follows that
    \begin{equation}
        \frac{\d H}{\d t}
        = \nabla H \cdot \frac{\d \vu}{\d t}
        = \nabla H^\top G \nabla H,
    \end{equation}
    and this is less than or equal to $0$ if $G \leq O$ and it vanishes if $G$ is skew-symmetric.
\end{proof}

\begin{proof}[Proof of Theorem \ref{thm:mass_conservation}]
    \begin{equation}
        \frac{\d}{\d t} \sum_k u_k = \vec{1} \frac{\d u}{\d t} = \vec{1} G \nabla H = 0.
    \end{equation}
\end{proof}

\begin{proof}[Proof of Theorem \ref{thm:discrete_conservation_dissipation}]
    If $G$ is negative semi-definite, it follows from the definition of the discrete gradient that
    \begin{equation}
        \displaystyle\frac{H(\vu^{(n+1)})-H(\vu^{(n)})}{t^{(n+1)}-t^{(n)}} = \bnabla H(\vu^{(n+1)}, \vu^{(n)})^\top \bG \bnabla H(\vu^{(n+1)}, \vu^{(n)}) \leq 0
    \end{equation}
    and the amount of the energy dissipation $\bnabla H(\vu^{(n+1)}, \vu^{(n)})^\top \bG \bnabla H(\vu^{(n+1)}, \vu^{(n)})$ is indeed an approximation to $\nabla H(\vu)^\top G \nabla H(\vu)$, which shows the discrete energy dissipation law. Similarly, if $G$ is skew-symmetric,
    \begin{equation}
        \displaystyle\frac{H(\vu^{(n+1)})-H(\vu^{(n)})}{t^{(n+1)}-t^{(n)}} = \bnabla H(\vu^{(n+1)}, \vu^{(n)})^\top \bG \bnabla H(\vu^{(n+1)}, \vu^{(n)}) = 0
    \end{equation}
    is obtained in the same way.
    For the discrete mass conservation law,
    \begin{equation}
        \frac{\sum_k u_k^{(n+1)}-\sum_k u_k^{(n)}}{t^{(n+1)}-t^{(n)}}=\vec{1}\left(\frac{\vu^{(n+1)}-\vu^{(n)}}{t^{(n+1)}-t^{(n)}}\right)=\vec{1}\ \!\bG \bnabla H(\vu^{(n+1)}, \vu^{(n)})=0.
    \end{equation}
\end{proof}

\begin{proof}[Proof of Theorem \ref{thm:discrete_differential_rules}]
    The first equation comes from the linearity with respect to the first argument.
    The third equation is obvious by definition.
    The second equation is a well-known result based on the studies of the discrete gradient methods~\cite{Furihata2010}:
    \begin{equation}
        \textstyle f(x_1)g(x_1)-f(x_2)g(x_2)=\frac{g(x_1)+g(x_2)}{2}(f(x_1)-f(x_2))+\frac{f(x_1)+f(x_2)}{2}(g(x_1)-g(x_2)).
    \end{equation}
\end{proof}

\section{Semi-Discretization of the Partial Differential Equations}\label{appendix:semi-discretization}
Although for simplicity we have only considered the ODEs in Section \ref{sec:method}, the target equations of our approach include PDEs.

The target PDEs are equations of the following form.
\begin{equation}\label{eq:target_pde}
    \frac{\partial u}{\partial t} = G(u) \nabla H,
\end{equation}
where $u$ may depend on $t$ and $x \in \mathbb{R}^n$, and $G(u)$ is a linear operator that depends on the function $u$. In the underlying functional space that admits the inner product $\langle \cdot, \cdot, \rangle$, we can consider the adjoint operator of $G(u)$ as the operator that satisfies
\begin{equation}
    \left \langle {G} v, w \right \rangle = \left \langle v, \tilde{G} w \right \rangle
\end{equation}
for any functions $v, w$. If the space is real and finite dimensional and the inner product is the standard inner product, the adjoint operator of a matrix $G$ is $G^\top$.

Similarly to the finite dimensional equations, it follows from the relation between the Fr\'echet derivative and the gradient
\begin{equation}
    \frac{\d H}{\d t} = \d H\left( \frac{\partial u}{\partial t}\right)
    = \left\langle \nabla H, \frac{\partial u}{\partial t} \right\rangle
    = \langle \nabla H, G \nabla H \rangle.
\end{equation}
Thus, Eq.~\eqref{eq:target_pde} has the energy dissipation law
\begin{equation}
    \frac{\d H}{\d t} = \langle \nabla H, G \nabla H \rangle
    \leq 0
\end{equation}
if the operator $G(u)$ is negative semi-definite in the sense that $\langle v, G v \rangle\leq 0$
for all $v$, and the energy conservation law
\begin{equation}
    \frac{\d H}{\d t} = 0
\end{equation}
if $G(u)$ is skew-adjoint; $\tilde{G}(u) = - G(u)$.

Examples of the negative semi-definite $G$ in practical applications include
\begin{equation}\label{eq:g_operator}
    G = (-1)^{s-1} \frac{\partial^{2 s}}{\partial x^{2 s}},
\end{equation}
where $s$ is a non-negative integer and we define $G = -1$ for $s=0$.
This operator is negative semi-definite with respect to the $L^2$ inner product under certain boundary conditions. For example, suppose that the underlying space is the interval $[0, 1]$. Then for functions $f$, $g$ it holds that
\begin{align*}
    \int_0^1 f \frac{\partial^2 g}{\partial x^2} \d x
    =
    - \int_0^1 \frac{\partial f}{\partial x} \frac{\partial g}{\partial x} \d x
    + \left[  f \frac{\partial g}{\partial x} \right]^1_0
    =
    \int_0^1 \frac{\partial^2 f}{\partial x^2} g \d x
    + \left[  f \frac{\partial g}{\partial x} \right]^1_0
    - \left[  \frac{\partial f}{\partial x} g \right]^1_0
\end{align*}
and hence if $f$ and $g$ satisfy
\begin{align*}
    \left[  f \frac{\partial g}{\partial x} \right]^1_0
    - \left[  \frac{\partial f}{\partial x} g \right]^1_0 = 0,
\end{align*}
the operator $\partial^2 / \partial x^2$ is negative semi-definite with respect to the $L^2$ inner product:
\begin{equation}
    \left \langle \frac{\partial^2 f}{\partial x^2}, g \right \rangle
    = - \left \langle f, \frac{\partial^2 g}{\partial x^2} \right \rangle.
\end{equation}
Similarly, the operators
\begin{equation}\label{eq:g_operator_skew}
    G = \frac{\partial^{2 s-1}}{\partial x^{2 s-1}},
\end{equation}
are skew-adjoint.

In order to apply our framework to the PDEs, it would be preferable to discretize the operator $G(u)$ because the property of this operator is essential for the energy conservation or the dissipation law and the mass conservation law.
If $G(u)$ is given by Eq.~\eqref{eq:g_operator} or Eq.~\eqref{eq:g_operator_skew}, this operator can be discretized by using the central difference operators such as Eq.~\eqref{eq:first_second_order_difference}, while preserving the desired properties.
In general, suppose that the matrix $D_s$ approximates the differential operator $\partial^s /\partial x^s$ then
\begin{equation}
    \frac{D_s + D_s^\top}{2}, \qquad \frac{D_s - D_s^\top}{2}
\end{equation}
are in principle respectively negative/positive semi-definite and skew-adjoint approximations to  $\partial^s /\partial x^s$. This design of the operator $G$ yields an approximation to the operator $G$ with the desired property and also with the desired accuracy. For further details for the structure-preserving semi-discretization, see,. e.g.,~\cite{Celledoni2012, Furihata2010}.

\section{Comparison with Symplectic Integrators}\label{appendix:comparison_with_symplectic}
The Ge--Marsden theorem shows that no method can be both symplectic and strictly energy-preserving~\cite{Marsden1988}.
Hence, the proposed discrete gradient method does not conflict with but complements a neural network model solved by a symplectic integrator~\cite{Chen2020a,Saemundsson2020,Zhong2020}.
One can choose a preferable one depending on targeted tasks.

Some symplectic integrators (such as the variational integrator) are known to preserve the momentum in a physical system.
The proposed discrete gradient method can have the property of conjugate symplecticity, which guarantees the preservation of the momentum with high accuracy~\cite{Hairer2010}; in fact for a certain class of problems the discrete gradient derived by the proposed algorithm is equivalent to the average vector field method, which is known to be conjugate symplectic of order four~\cite{Celledoni2012}.
Moreover, if a certain quantity other than the energy should be strictly preserved, one can design another discrete gradient by the method proposed in \cite{Dahlby2011} so that both the energy and the quantity are preserved.

With varying time-step, symplectic integrators are in general known to lose symplecticity and cannot preserve the system energy just like non-symplectic Runge--Kutta methods~\cite{Hairer2006}, while the proposed discrete gradient method can.
In particular, depending on the learned parameters, the proposed method can choose the time-step while preserving the energy.

\section{Automatic Discrete Differentiation Algorithm}\label{app:add}
Many practical implementations of the automatic differentiation algorithm indeed obtain gradients directly~\cite{Paszke2017}; hence, they are sometimes called the \textsf{autograd} algorithm.
A discrete version of the autograd algorithm is enough for this study.
From this viewpoint, the implementation of the discrete autograd algorithm is introduced as follows.

The Fr\'echet derivative $\d g(\cdot;\vu):\R^N\rightarrow\R^M$ of a function $g:\R^N\rightarrow\R^M$ at $\vu$ is a bounded linear operator that satisfies the following condition.
\begin{equation}
    \lim_{||\vh||\rightarrow+0}\frac{||g(\vu+\vh)-g(\vu)-\d g(\vh;\vu)||}{||\vh||}=0.
\end{equation}
The Fr\'echet derivative $\d g$ can be regarded as multiplication by the Jacobian matrix $J_g(\vu)\in\R^{M\times N}$ at $\vu$.
\begin{equation}
    \d g(\vw;\vu) = J_g(\vu)\vw.
\end{equation}
The chain-rule can be rewritten as a chain of Jacobian matrices.
\begin{equation}
    \d (f\circ g)(\vw;\vu) = \d f(\cdot;g(\vu))\circ\d g(\vw;\vu)= J_f(g(\vu)) J_g(\vu) \vw.
\end{equation}
The gradient $\nabla f$ of a scalar-valued function $f:\R^M\rightarrow\R$ is defined using an inner product $\cdot$ as follows.
\begin{equation}
    \lim_{||\vk||\rightarrow+0}\frac{||f(\vv+\vk)-f(\vv)-\nabla f(\vv)\cdot\vk||}{||\vk||}=0.
\end{equation}
This implies that the gradient is dual to the derivative.
Hence, the gradient $\nabla f$ is equal to the transposed Jacobian matrix $J_f$.
\begin{equation}
    \d f(\vw;\vv)= J_f(\vv)\vw =\nabla f(\vv)\cdot \vw.
\end{equation}
Therefore, the gradient $\nabla (f\circ g)$ of the compositional function $f\circ g$ is obtained by multiplying the upper-layer gradient $\nabla f$ by the transposed lower-layer Jacobian matrix $J_g(\vu)^\top$.
\begin{equation}
    \nabla (f\circ g)(\vu) = J_g(\vu)^\top\nabla f(g(\vu)).
\end{equation}
This is the autograd algorithm to obtain the gradient of a compositional function, which is shown in the left panel of Fig.~\ref{fig:add}.

\begin{figure}[t]\centering
    \includegraphics[page=2,scale=0.4]{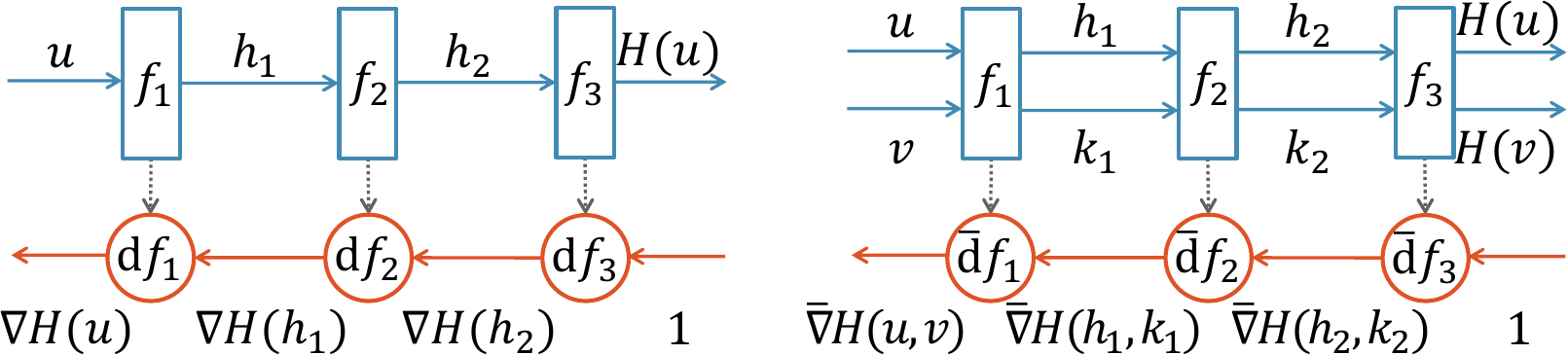}
    \caption{A conceptual comparison between the autograd algorithm (left) and the discrete autograd algorithm (right).}
    \label{fig:add}
\end{figure}

\begin{figure}[t]
    \begin{algorithm}[H]
        \footnotesize
        \caption{Discrete Autograd Algorithm}
        \label{alg:add}
        \begin{algorithmic}
            \renewcommand{\algorithmicrequire}{\textbf{Input:}}
            \renewcommand{\algorithmicensure}{\textbf{Output:}}
            \Require a function $H(\cdot)=f_N\circ\cdots\circ f_1(\cdot)$ for $i\in\{1,\dots, N\}$, and arguments $\vu$ and $\vv$
            \Ensure discrete gradient $\bnabla H(\vu,\vv)$
            \State $\vh_0\leftarrow\vu$, $\vk_0\leftarrow\vv$
            \For{$i=1,\cdots,N$}
            \State $\vh_i\leftarrow f_i(\vh_{i-1})$, $\vk_i\leftarrow f_i(\vk_{i-1})$
            \EndFor
            \State $H(\vu)\leftarrow \vh_N$, $H(\vv)\leftarrow\vk_N$
            \State $\bnabla H(\vh_N,\vk_N)\leftarrow 1\!\!1$
            \For{$i=N,\cdots,1$}
            \If{$f_i$ is a linear layer}
            \State $\bar J_{f_i}\leftarrow J_{f_i}$
            \ElsIf{$f_i$ is an element-wise nonlinear activation function}
            \State $\bar J_{f_i}\leftarrow \mathrm{diag}(\vh_{i}-\vk_{i})\mathrm{diag}(\vh_{i-1}-\vk_{i-1})^{-1}$
            \Else
            \State $\bar J_{f_i}\leftarrow$(depending on the function)
            \EndIf
            \State $\bnabla H(\vh_{i-1},\vk_{i-1})\leftarrow\bar J_{f_i}^\top\bnabla H(\vh_i,\vk_i) $
            \EndFor
            \State \Return $\bnabla H(\vu,\vv)\leftarrow\bnabla H(\vh_0,\vk_0)$
        \end{algorithmic}
    \end{algorithm}
\end{figure}

For the discrete autograd algorithm, we replace the Jacobian matrices $J_g$ with their discrete counterparts $\bar J_g$ as shown in the right panel of Fig.~\ref{fig:add}.
For a linear layer $g$, the discrete Jacobian matrix $\bar J_g$ is equal to the ordinary one $J_g(\vv)=\frac{\partial g}{\partial \vv}$.
For an element-wise nonlinear activation layer $g$, the discrete Jacobian matrix $\bar J_g$ is a diagonal matrix where each non-zero element is expressed as $\frac{f(v_1)-f(v_2)}{v_1-v_2}$ when given two scalar arguments $v_1$ and $v_2$.
If the two arguments $v_1$ and $v_2$ are closer than $\epsilon$, we use the gradient $\frac{\d f}{\d z}$ at the midpoint $z=\frac{v_1+v_2}{2}$ to avoid the loss of significance.
We empirically found that $\epsilon=10^{-6}$ and $\epsilon=10^{-12}$ worked well with single and double precisions, respectively.
We summarize the discrete autograd algorithm in Algorithm~\ref{alg:add}.

For obtaining a discrete gradient, the automatic discrete differentiation algorithm requires two forward paths and one modified backward path.
The computational cost of a discrete gradient is one and a half times as much as that of the ordinary gradient.
The training of DGNet is less expensive than the training of HNN with the explicit midpoint method (RK2), and it is tens of times less expensive than the training of HNN with the adaptive Dormand--Prince method.

\section{Details of Datasets and Results}\label{appendix:datasets_results}
We implemented all codes using Python v3.7.3 with libraries; numpy v1.16.2, scipy v1.2.1, and PyTorch v1.4.0~\cite{Paszke2017}.
We performed all experiments on NVIDIA TITAN V for PDEs and GeForce 2080 Ti for ODEs.
We also used torchdiffeq v0.0.1 library for numerical integrations of neural network models~\cite{Chen2018e}.
The results are summarized in Tables~\ref{tab:pde_score_appendix}, \ref{tab:ode_score_appendix}, and \ref{tab:real_score_appendix}.

\begin{figure}[t]\centering
    \includegraphics[scale=1.0]{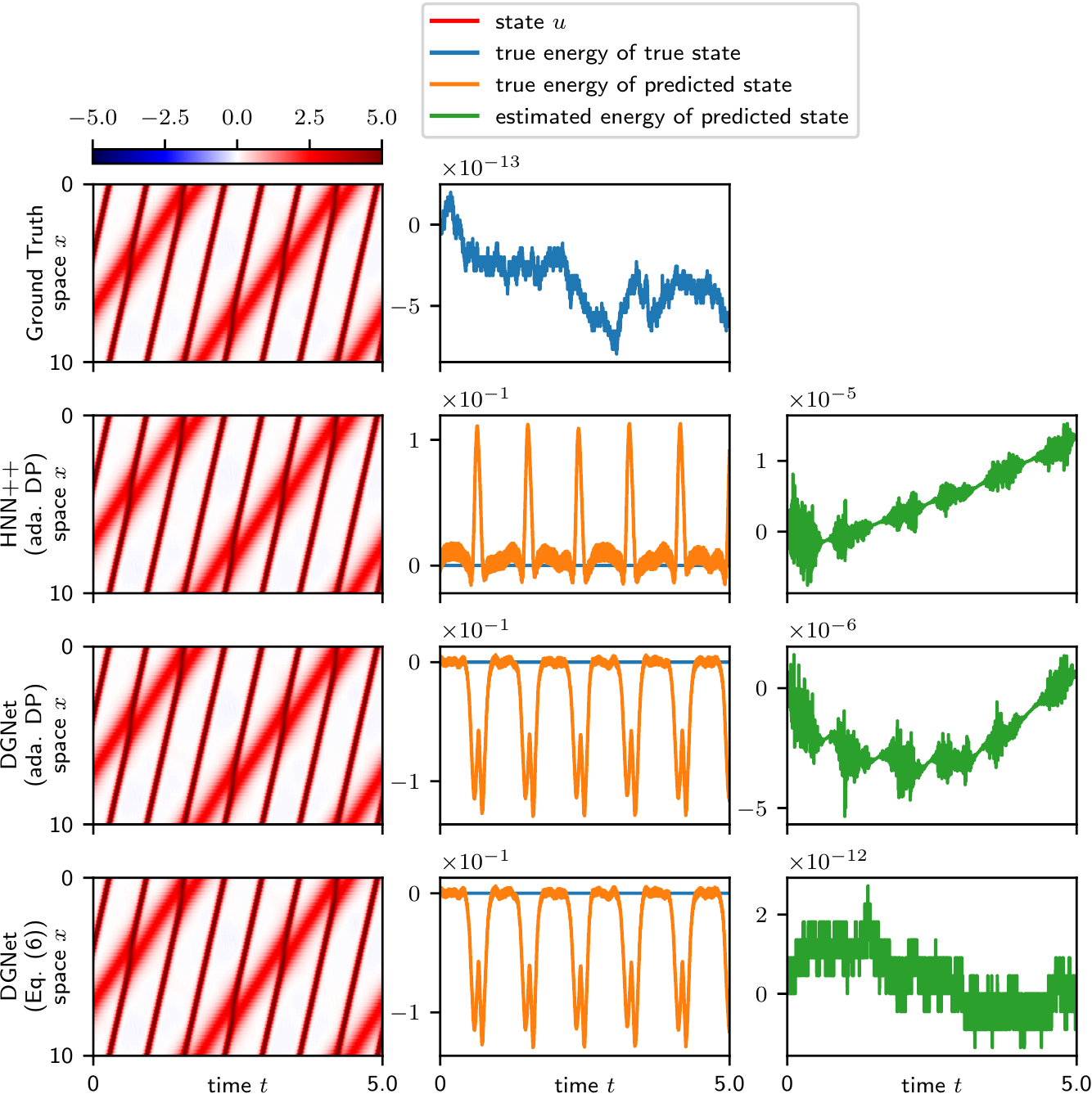}\\
    \hspace*{8mm}(a)\hspace*{40mm}(b)\hspace*{40mm}(c)
    \caption{
        Results of long-term predictions.
        Models and integrators are on the left edge.
        (a) State $u$.
        (b) The system energies of the true state and predicted state derived from the true equation, denoted by blue and orange lines, respectively.
        (c) The system energy of the predicted state when considering the trained neural network as a true equation.
    }
    \label{fig:longprediction}
\end{figure}
\textbf{PDE systems.}
We provide detailed results on Tables~\ref{tab:pde_score_appendix}, which corresponds to Table~\ref{tab:pde_score}.
The standard deviations are preceded by plus-minus signs $\pm$, and the scales are noted in brackets.
For all errors, the standard deviations are large.
The performence difference between HNN\pp\ and DGNet would not be significant for the KdV equation but it is obviously significant for the Cahn--Hilliard equation.

We also provide the results of longer-time predictions of the KdV equation in Fig.~\ref{fig:longprediction}.
Under the conditions same as in Section~\ref{sec:experiments}, HNN\pp\ and DGNet predicted state $u$ for 5,000 steps, as shown in the left column.
In the center column, the true equation in Eq.~\eqref{eq:kdv} gives the system energies of the true state and predicted state, as denoted by the blue and orange lines, respectively.
The system energy of the true state is conserved within a range of the rounding error.
The prediction errors become larger when two solitons collide with each other, but they are restored to their former levels; each model learned collisions qualitatively rather than quantitatively.
In the right column, each panel shows the system energy of the predicted state when a trained neural network model is considered as a true equation.
The neural network models formed Hamiltonian systems in Eq.~\eqref{eq:gradient_flow} and may conserve the system energy.
However, the system energy learned by HNN\pp\ increases, implying that the conservation law is destroyed.
This phenomena is called \emph{energy drift} and occurs commonly in a Runge--Kutta method integrating a Hamiltonian system.
The energy drift is a practical issue for a simulation of molecular dynamics and solar systems, where the number of time steps is more than one million and the numerical error becomes more significant than the modeling error.
This is the main reason why structure-preserving integrators are needed~\cite{Hairer2006}.
The Runge--Kutta method also destroys the conservation law that DGNet potentially produces (see the third row).
Only when using the implicit scheme in Eq.~\eqref{eq:discrete_system}, DGNet gives the system energy that fluctuates within a range of $\pm3\times10^{-12}$, implying that DGNet conserves the system energy only with the rounding error.

\textbf{ODE systems.}
In Section~\ref{sec:experiments}, we evaluated the models on the ODE datasets from the original source codes for the study on the HNN~\cite{Greydanus2019}.
Each of their datasets is composed of three parts; the first is for training, the second is for evaluating the time-derivative error, and the third is for evaluating the accuracy of the long-term prediction.
The number of observations and the duration of each trajectory are summarized in Table~\ref{tab:dataset}.
For the pendulum dataset, each trajectory in the training and test sets consists of 45 observations over three unit times. Meanwhile for the long-term prediction set, it was composed of 100 observations over 20 unit times.
This difference did not matter in the original study because each model was trained with the true time-derivative and it used the adaptive Dormand--Prince method for the time-series prediction.
Conversely, in our experiments, the finite difference was instead given.
The difference in the time step size caused unanticipated impacts on all models.
Hence, in Section~\ref{sec:experiments}, we set the conditions for training and testing to the same as those for the long-term prediction.
Moreover, for the spring dataset prediction, the time step size was rescaled to compensate for the observation noise in the original implementation.
However, we found that this modification did not matter in our experiments; thus, we removed the rescaling.

We provide detailed results on Tables~\ref{tab:ode_score_appendix} and~\ref{tab:real_score_appendix}, which correspond to Table~\ref{tab:ode_score}.
The standard deviations are preceded by plus-minus signs $\pm$, and the scales are noted in brackets.
For the real pendulum dataset, the long-term prediction error of DGNet is significantly smaller than that of the HNN; DGNet extimates the amount of energy dissipation well, as expected in Theorem~\ref{thm:discrete_conservation_dissipation}.
Even through other differences between the HNN and DGNet would not be significant, this fact is still remarkable.
For the HNN, the Dormand--Prince method is a fourth--order method and adjusts the time step size to suppress the prediction error smaller than a given threshold; it is the most numerically accurate and realible integrator in our experiments.
Conversely, DGNet and the leapfrog integrator are second--order methods.
For learning physical phenomena, the qualitative property is important equally to or more than the quantitative accuracy.

\begin{figure}[t]\centering
    \includegraphics[scale=1.0]{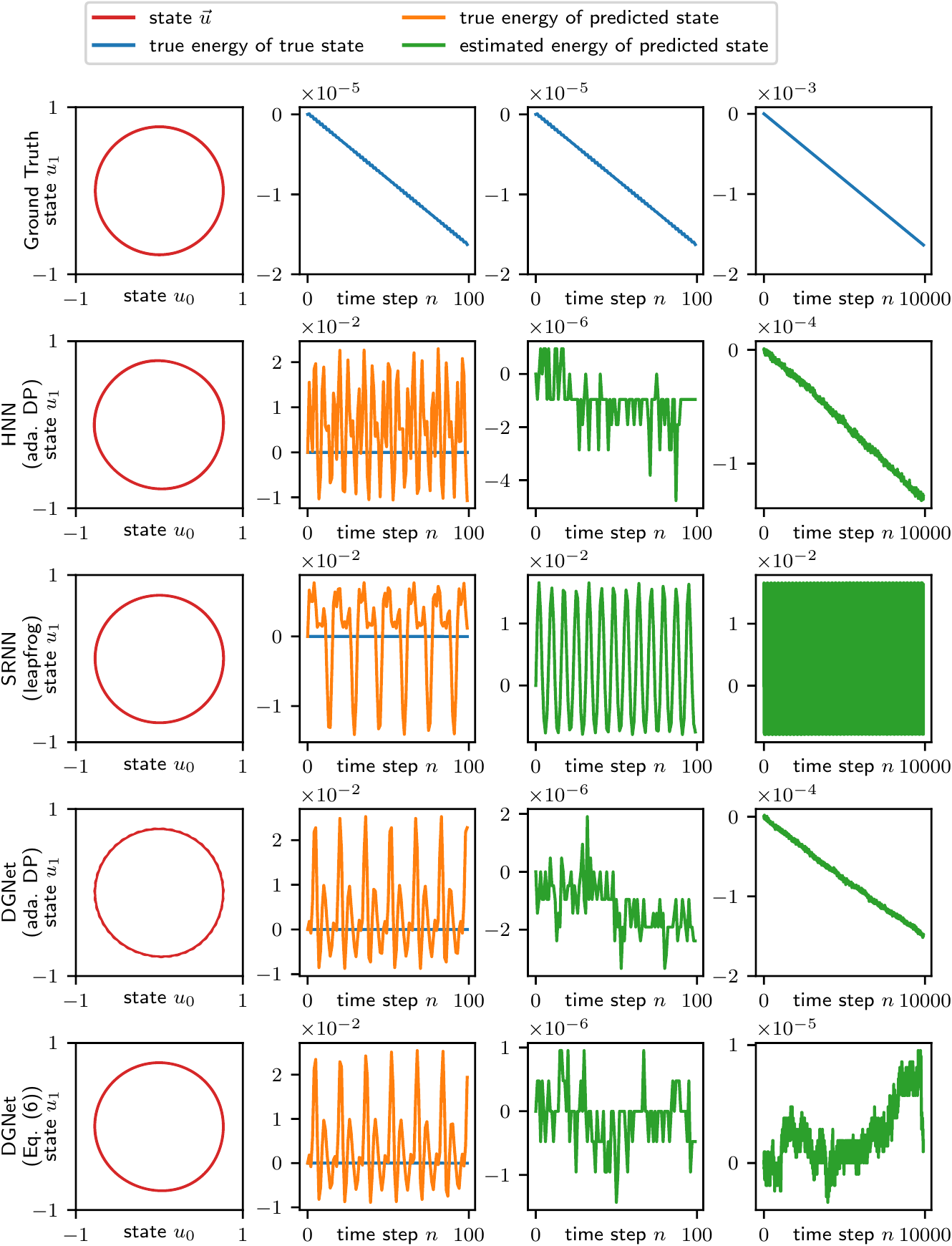}
    \hspace*{10mm}(a)\hspace*{25mm}(b)\hspace*{30mm}(c)\hspace*{30mm}(d)
    \caption{
        Detailed results of the mass-spring dataset.
        Models and integrators are on the left edge.
        (a) A trajectory of state $u$.
        (b) The system energies of the true state and predicted state derived from the true equation, denoted by blue and orange lines, respectively.
        (c)(d) The system energy of the predicted state when considering the trained neural network as a true equation (c) for 100 steps, and (d) for 10,000 steps.
    }
    \label{fig:spring}
\end{figure}

\begin{figure}[t]\centering
    \includegraphics[scale=1.0]{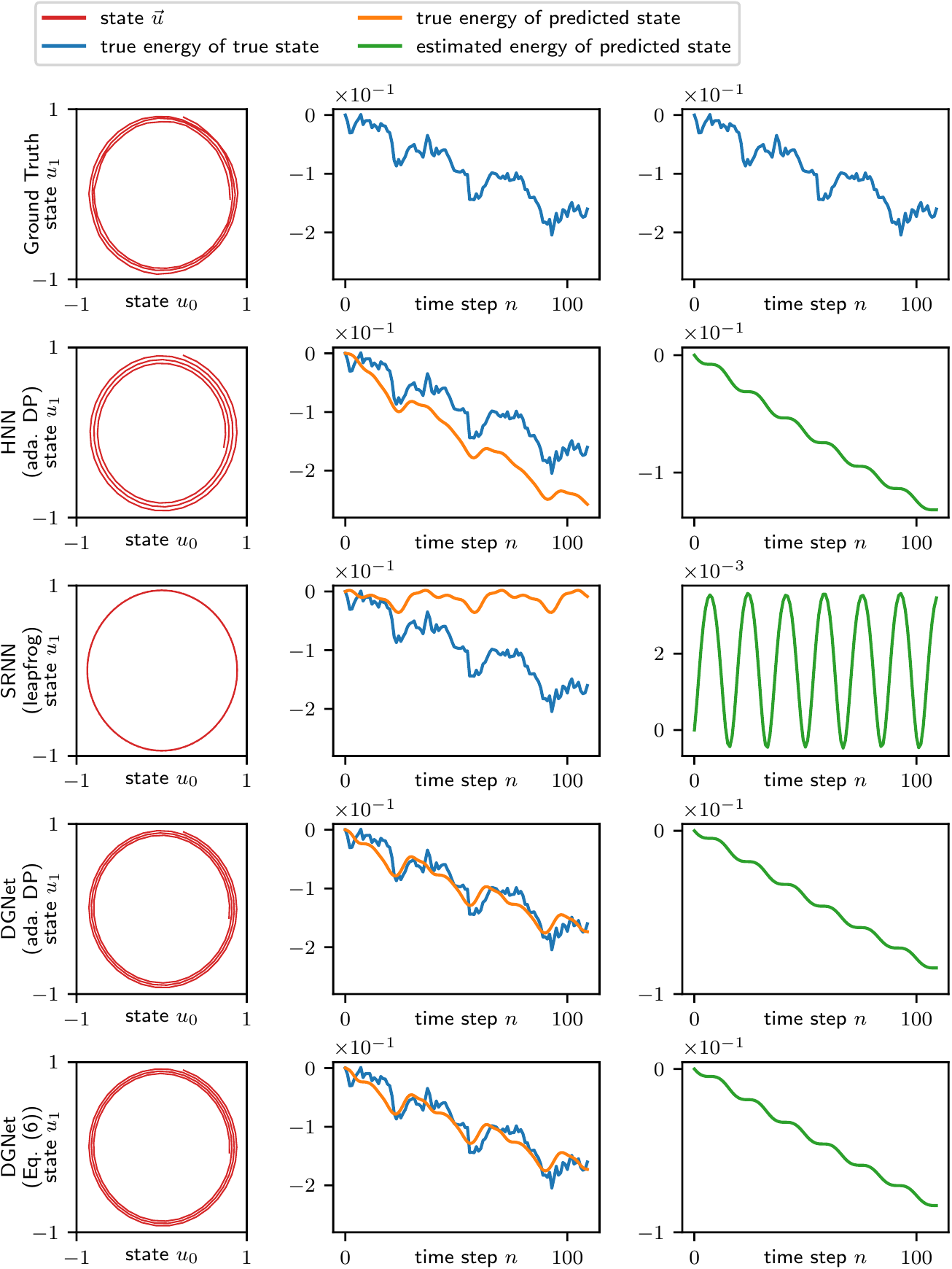}
    \hspace*{5mm}(a)\hspace*{40mm}(b)\hspace*{45mm}(c)
    \caption{
        Detailed results of the real pendulum dataset.
        Models and integrators are on the left edge.
        (a) A trajectory of state $u$.
        (b) The system energies of the true state and predicted state derived from the true equation, denoted by blue and orange lines, respectively.
        (c) The system energy of the predicted state when considering the trained neural network as a true equation.
    }
    \label{fig:real2}
\end{figure}

We provide the detailed results of the mass-spring dataset in Fig.~\ref{fig:spring}.
The leftmost column shows trajectories of state $\vu$, each of which forms a circle.
In the second left column, the true equation gives the system energy of the true state and predicted state.
Because the ground truth data was generated using a Runge--Kutta method (specifically, the adaptive Dormand--Prince method implemented in \textsf{solve\_imp} method of scipy library), the true system energy is drifting.
For all neural network models, the system energy was fluctuating over a wide range due to the modeling error.
Each of the remaining panels shows the system energy of the predicted state when a trained neural network model is considered as a true equation for 100 steps in the second right column and 10,000 steps in the rightmost column.
Using the Runge--Kutta method (specifically, the adaptive Dormand--Prince method implemented in \textsf{odeint} method of torchdiffeq library), the system energies of HNN\pp\ and DGNet are drifting.
Using the leapfrog integrator, the system energy of SRNN is fluctuating over the widest range.
The leapfrog integrator is a symplectic integrator and conserves the ``shadow'' Hamiltonian, which is an approximation to the true Hamiltonian~\cite{Hairer2006}.
The fluctuation makes it difficult to evaluate the energy efficiency of a system, and this is a main drawback of symplectic integrators.
Using the implicit scheme in Eq.~\eqref{eq:discrete_system}, the system energy of DGNet fluctuates within the narrowest range, demonstrating the superiority of the discrete gradient method.

We provide the detailed results of the real rendulum dataset in Fig.~\ref{fig:real2}.
The left column shows trajectories of state $\vu$.
Each trajectory forms a spiral except for the leapfrog integrator, which produces a circle without dissipation.
In the center column, the true equation gives the system energy of the true state and predicted state.
The true energy derived only from the angle and momentum of the pendulum bob does not monotonically decrease because it ignores the other components (e.g., the pendulum rod).
The right column shows the system energy of the predicted state when a trained neural network model is considered as a true equation.
The system energy estimated by HNN and DGNet dissipates monotonically (see green lines); these models build the alternative energy by implicitly estimating the other components.
For a real-world problem, we cannot always observe all states, and then, the ``true'' equation cannot describes the time evolution.
The data-driven modeling enables us to predict a partially observable state with the conservation and dissipation laws.

\begin{landscape}
    \begin{table}[t]
        \centering\small
        \caption{Detailed results on the PDE datasets corresponding to Table~\ref{tab:pde_score}.}\label{tab:pde_score_appendix}
        \begin{tabular}{lllllllll}
            \toprule
                                                   & \mcb{Integrator}               & \multicolumn{3}{c}{\textbf{KdV equation}} & \multicolumn{3}{c}{\textbf{Cahn--Hilliard equation}}                                                                                                                               \\
            \cmidrule(lr){2-3}\cmidrule(lr){4-6}\cmidrule(lr){7-9}
            \textbf{Model}                         & \textbf{Training}              & \textbf{Prediction}                       & \mca{Deriv.}                                         & \mca{Energy}            & \mca{Mass}              & \mca{Deriv.}          & \mca{Energy}           & \mca{Mass}             \\
            \midrule
            \multirow{2}{*}{NODE~\cite{Chen2018e}} & RK2                            & RK2                                       & \score{1.15}{0.01}{1}                                & \score{4.57}{3.62}{4}   & \score{2.86}{0.38}{0}   & \score{7.91}{0.03}{2} & \score{1.43}{0.05}{-2} & \score{9.15}{0.11}{-1} \\
                                                   & ada.~\!DP                      & ada.~\!DP                                 & \score{1.15}{0.01}{1}                                & \score{5.58}{6.56}{4}   & \score{2.84}{0.43}{0}   & \score{7.90}{0.03}{2} & \score{1.42}{0.05}{-2} & \score{9.14}{0.11}{-1} \\
            \midrule
            \multirow{2}{*}{HNN\pp$^*$}            & RK2                            & RK2                                       & \score{3.63}{4.14}{-2}                               & \score{6.32}{10.26}{-3} & \score{7.00}{10.80}{-4} & \score{3.44}{0.20}{2} & \score{1.33}{1.40}{-1} & \score{8.76}{5.07}{-2} \\
                                                   & ada.~\!DP                      & ada.~\!DP                                 & \score{2.33}{2.19}{-2}                               & \score{3.01}{4.81}{-3}  & \score{3.35}{4.92}{-4}  & \score{3.30}{0.75}{1} & \score{4.89}{1.74}{-6} & \score{7.95}{3.12}{-4} \\
            \midrule
                                                   & \upl{4mm}                      & RK2                                       & \upl{10mm}                                           & \score{1.84}{1.69}{-3}  & \score{2.78}{2.77}{-4}  & \upl{10mm}            & \score{6.61}{3.34}{0}  & \score{8.22}{2.04}{-1} \\
            DGNet                                  & Eq.~\eqref{eq:discrete_system} & ada.~\!DP                                 & \score{1.75}{1.00}{-2}                               & \score{1.60}{1.65}{-3}  & \score{2.54}{2.91}{-4}  & \score{7.14}{8.27}{0} & \score{3.39}{4.05}{-7} & \score{6.95}{7.64}{-5} \\
                                                   & \downl{4mm}                    & Eq.~\eqref{eq:discrete_system}            & \downl{10mm}                                         & \score{1.60}{1.65}{-3}  & \score{2.55}{3.00}{-4}  & \downl{10mm}          & \score{3.40}{4.10}{-7} & \score{6.96}{7.70}{-5} \\
            \bottomrule
        \end{tabular}
    \end{table}
    \begin{table}[t]
        \centering\small
        \caption{Detailed results on the ODE datasets corresponding to Table~\ref{tab:ode_score}.}\label{tab:ode_score_appendix}
        \begin{tabular}{lllllllll}
            \toprule
                                                      & \mcb{Integrator}               & \mcb{Mass--Spring}             & \mcb{Pendulum}         & \mcb{2-Body}                                                                                                                \\
            \cmidrule(lr){2-3}\cmidrule(lr){4-5}\cmidrule(lr){6-7}\cmidrule(lr){8-9}
            \textbf{Model}                            & \textbf{Training}              & \textbf{Prediction}            & \mca{Deriv.}           & \mca{Energy}           & \mca{Deriv.}           & \mca{Energy}           & \mca{Deriv.}           & \mca{Energy}            \\
            \midrule
            \multirow{2}{*}{NODE}                     & RK2                            & RK2                            & \score{5.27}{0.32}{-2} & \score{5.70}{1.53}{-1} & \score{5.67}{0.56}{-2} & \score{4.60}{0.76}{0}  & \score{2.08}{0.40}{-5} & \score{1.44}{1.96}{-1}  \\
                                                      & ada.~\!DP                      & ada.~\!DP                      & \score{5.57}{0.38}{-2} & \score{5.74}{1.51}{-1} & \score{5.54}{0.61}{-2} & \score{4.62}{0.76}{0}  & \score{2.07}{0.41}{-5} & \score{1.83}{3.05}{-1}  \\
            \midrule
            \multirow{2}{*}{HNN~\cite{Greydanus2019}} & RK2                            & RK2                            & \score{3.82}{0.09}{-2} & \score{6.13}{1.50}{-2} & \score{4.25}{0.24}{-2} & \score{4.04}{0.69}{-1} & \score{5.39}{2.65}{-6} & \score{9.39}{8.10}{-5}  \\
                                                      & ada.~\!DP                      & ada.~\!DP                      & \score{3.99}{0.09}{-2} & \score{1.74}{3.99}{-3} & \score{4.09}{0.29}{-2} & \score{1.66}{0.59}{-2} & \score{6.21}{4.65}{-6} & \score{8.18}{6.08}{-5}  \\
            \midrule
            SRNN~\cite{Chen2020a}                     & leapfrog                       & leapfrog                       & \score{3.95}{0.08}{-2} & \score{6.90}{8.08}{-4} & \score{3.92}{0.14}{-2} & \score{1.12}{0.67}{-2} & \score{4.36}{2.40}{-6} & \score{4.04}{3.22}{-5}  \\
            \midrule
                                                      & \upl{4mm}                      & RK2                            & \upl{10mm}             & \score{6.13}{0.74}{-2} & \upl{10mm}             & \score{7.43}{1.23}{-1} & \upl{10mm}             & \score{8.11}{10.91}{-5} \\
            DGNet                                     & Eq.~\eqref{eq:discrete_system} & ada.~\!DP                      & \score{3.85}{0.09}{-2} & \score{6.16}{4.28}{-4} & \score{3.93}{0.19}{-2} & \score{1.61}{1.16}{-2} & \score{7.80}{4.22}{-6} & \score{8.10}{10.91}{-5} \\
                                                      & \downl{4mm}                    & Eq.~\eqref{eq:discrete_system} & \downl{10mm}           & \score{6.17}{4.28}{-4} & \downl{10mm}           & \score{1.08}{0.91}{-2} & \downl{10mm}           & \score{8.10}{10.91}{-5} \\
            \bottomrule
        \end{tabular}
    \end{table}

\end{landscape}

\begin{table}[t]
    \centering\small
    \caption{Detailed results on the real pendulum dataset corresponding to Table~\ref{tab:ode_score}.}\label{tab:real_score_appendix}
    \begin{tabular}{lllll}
        \toprule
                                                  & \mcb{Integrator}               & \mcb{Real Pendulum}                                                              \\
        \cmidrule(lr){2-3}\cmidrule(lr){4-5}
        \textbf{Model}                            & \textbf{Training}              & \textbf{Prediction}            & \mca{Deriv.}           & \mca{Energy}           \\
        \midrule
        \multirow{2}{*}{NODE}                     & RK2                            & RK2                            & \score{1.38}{0.02}{-3} & \score{6.22}{4.32}{-4} \\
                                                  & ada.~\!DP                      & ada.~\!DP                      & \score{1.37}{0.02}{-3} & \score{5.88}{4.02}{-4} \\
        \midrule
        \multirow{2}{*}{HNN~\cite{Greydanus2019}} & RK2                            & RK2                            & \score{1.42}{0.22}{-3} & \score{2.86}{0.50}{-3} \\
                                                  & ada.~\!DP                      & ada.~\!DP                      & \score{1.41}{0.15}{-3} & \score{3.44}{1.71}{-3} \\
        \midrule
        SRNN~\cite{Chen2020a}                     & leapfrog                       & leapfrog                       & \score{1.38}{0.02}{-3} & \score{9.63}{0.49}{-3} \\
        \midrule
                                                  & \upl{4mm}                      & RK2                            & \upl{10mm}             & \score{8.63}{4.96}{-4} \\
        DGNet                                     & Eq.~\eqref{eq:discrete_system} & ada.~\!DP                      & \score{1.38}{0.06}{-3} & \score{4.92}{3.94}{-4} \\
                                                  & \downl{4mm}                    & Eq.~\eqref{eq:discrete_system} & \downl{10mm}           & \score{5.04}{3.99}{-4} \\
        \bottomrule
    \end{tabular}
\end{table}

\begin{table}[t]
    \caption{Details of the datasets used in~\cite{Greydanus2019}}
    \label{tab:dataset}
    \footnotesize\centering
    \begin{tabular}{lR{12mm}R{12mm}R{12mm}R{10mm}R{12mm}R{12mm}R{10mm}}
        \toprule
                         &                      & \mcc{Training/Test} & \mcc{Long-Term Prediction}                                                \\
        \cmidrule(lr){3-5}
        \cmidrule(lr){6-8}
        \textbf{Dataset} & \textbf{\#Iteration} & \#Traject.          & \#Observ.                  & Duration & \#Traject. & \#Observ. & Duration \\
        \midrule
        Mass-Spring      & 2,000                & 25/25               & 30                         & 3        & 15         & 100       & 20       \\
        Pendulum         & 2,000                & 25/25               & 45                         & 3        & 15         & 100       & 20       \\
        2-Body           & 10,000               & 800/200             & 50                         & 20       & 15         & 500       & 25       \\
        \bottomrule
    \end{tabular}
\end{table}

\end{document}